\documentclass{gtmon_a}
\pdfoutput=1
\usepackage{pinlabel}

\proceedingstitle{Heegaard splittings of 3--manifolds (Haifa 2005)}
\conferencestart{10 July 2005}
\conferenceend{19 July 2005}
\conferencename{Heegaard splittings of 3--manifolds}
\conferencelocation{Haifa}

\editor{Cameron Gordon}
\givenname{Cameron}
\surname{Gordon}

\editor{Yoav}
\givenname{Yoav}
\surname{Moriah}

\title{Thin position for knots and 3--manifolds: a unified approach}

\author{Hugh Howards}
\givenname{Hugh}
\surname{Howards}
\address{Department of Mathematics\\
Wake Forest University\\\newline
PO Box 7388\\
127 Manchester Hall\\
Winston-Salem NC 27109\\
USA}
\email{howards@wfu.edu}
\urladdr{}

\author{Yo'av Rieck}
\givenname{Yo'av}
\surname{Rieck}
\address{Department of Mathematics\\
University of Arkansas\\\newline
Fayetteville AR 72701\\
USA}
\email{yoav@uark.edu}
\urladdr{}

\author{Jennifer Schultens}
\givenname{Jennifer}
\surname{Schultens}
\address{Department of Mathematics\\
University of California, Davis\\\newline
1 Shields Avenue\\
Davis CA 95616\\
USA}
\email{jcs@math.ucdavis.edu}
\urladdr{}

\volumenumber{12}
\issuenumber{}
\publicationyear{2007}
\papernumber{04}
\startpage{89}
\endpage{120}

\doi{}
\MR{}
\Zbl{}

\arxivreference{} 


\received{11 March 2006}
\revised{20 January 2007}
\accepted{20 January 2007}
\published{3 December 2007}
\publishedonline{3 December 2007}
\proposed{}
\seconded{}
\corresponding{}
\version{}

\makeatletter
\def\cnewtheorem#1[#2]#3{\newtheorem{#1}{#3}[section]
\expandafter\let\csname c@#1\endcsname\c@thm}

\let\xysavmatrix\xymatrix
\def\xymatrix{\disablesubscriptcorrection\xysavmatrix}
\AtBeginDocument{\let\bar\wbar\let\tilde\wtilde\let\hat\what}

\makeop{height}

\newtheorem{thm}{Theorem}
\cnewtheorem{lem}[thm]{Lemma}
\cnewtheorem{prop}[thm]{Proposition}
\cnewtheorem{cor}[thm]{Corollary}

\theoremstyle{remark}
\cnewtheorem{defn}[thm]{Definition}
\cnewtheorem{rem}[thm]{Remark}
\cnewtheorem{meta}[thm]{Meta-theorem}

\makeatother  

\newcommand{\mdc}{manifold decomposition}
\newcommand{\dbc}{2--fold branched cover}

\begin{document}

\begin{abstract}
We unify the notions of thin position for knots and for 3--manifolds
and survey recent work concerning these notions.
\end{abstract}

\maketitle

\section{Introduction}

Thin position for knots and for 3--manifolds have become basic tools
for 3--manifold topologists and knot theorists.  When David Gabai first
introduced the notion of thin position for knots as an ad hoc tool in
studying foliations of 3--manifolds he may not have foreseen the
widespread interest this notion would engender.  Thin position for
knots featured prominently in the work of Mark Culler, Cameron McA
Gordon, John Luecke and Peter Shalen concerning Dehn surgery on knots
as well in the proof by Cameron McA Gordon and John Luecke that knots
are determined by their complements.  It also played a crucial role in
Abigail Thompson's proof that there is an algorithm to recognize
${\mathbb S}^3$; Rubinstein's original argument \cite{rubinstein} used the
related concept of {\it minimax sweepouts} and {\it normal surfaces}.

A knot in thin position appears to be ideally situated from many
points of view.  This is demonstrated, for instance, by the work of
Daniel J. Heath and Tsuyoshi Kobayashi.  There is also a growing
expectation that some knot invariants can be calculated most
efficiently by employing thin position.

Later, Martin Scharlemann and Abigail Thompson introduced a related,
but not completely analogous, notion of thin position for 3--manifolds.
At first glance, their theory appeared elegant but of little use.  It
took a number of years for the strength of their theory to come to
fruition.  This theory has now become one of the fundamental tools in
the study of 3--manifolds.  Moreover, it has proved more natural than
the notion of thin position for knots.  This has prompted Martin
Scharlemann and Abigail Thompson to begin reworking the notion of thin
position for knots under the guise of ``slender knots''.  Their work
is beyond the scope of this article.

The aim of this article is to introduce the novice to the notion of
thin position for knots and 3--manifolds.  The emphasis here is to
underline the formal analogy of the definitions.  Each of these
notions is defined more naturally elsewhere.  For the most natural
definition of thin position for knots, see Gabai \cite{G}.  And for a more
extensive treatment of thin position for knots, see Scharlemann \cite{Scharl}.
For the most natural definition of thin position for 3--manifolds, see
Scharlemann and Thompson
\cite{ST}.  The added formality here is designed to unify the two
definitions.  This should allow an easy adaptation of the underlying
framework to numerous other settings.  In this paper we avoid some of the more
technical details; for an extensive introduction to the subject see
Saito, Scharlemann and Schultens \cite{saito}. 

We wish to thank Dave Bayer for suggesting this project, Marty Scharlemann
for helpful discussions, and the referee for many helpful suggestions.

\section{Thin position}

To define thin position in a general setting, we need the following: A pair of
manifolds $(N, M)$ with $N \subset M$.  A constraint $C$ that may be placed on
the set, ${\cal 
M}$, of Morse functions on $(N, M)$.  A function $g\co {\cal L} \rightarrow
{\mathbb R}^{\infty}$, for ${\cal L}$ the set of ordered pairs of level sets
of the elements of ${\cal M}$.  A well ordered  set ${\cal O}$.  And finally,
a function $f\co  {\mathbb R}^\infty \rightarrow {\cal O}$.  
We note that $g$ maps into $\mathbb{R}^n$ (for some $n$ that
depends on the manifold and the knot); we identify $\mathbb{R}^n$ with
$\mathbb{R}^\infty$ with all but the first $n$ coordinates set to zero.
Intuitively, $g$ measures the complexity of individual levels and $f$
measures the complexity of $(N,M)$. 

\begin{rem}
In fact, the definition can be made a little more general, as $N$ does not
need to be a manifold.  As an example, below we discuss a few settings were
$N$ is a graph.  
\end{rem}

Let $(N, M)$, $C$, $g$, ${\cal O}$ and $f$ be as required.
Set

\[{\cal C} = \{ h \in {\cal M} \; | \; h \; \mbox{satisfies} \; C\; \}.  \]

Given $h \in {\cal C}$, denote the critical values of $h$, in
increasing order, by $c_0, \dots, c_n$.  Note that since $h$ is a
Morse function on pairs, a critical value of $h$ is a critical value
either of $h|_N$ or of $h|_M$.  For $i = 1, \dots, n$, choose a
regular value $r_i$ such that $c_{i-1} < r_i < c_i$.  Consider the
finite sequence
\[(h|_N^{-1}(r_1), h|_M^{-1}(r_1))\dots, (h|_N^{-1}(r_n), h|_M^{-1}(r_n))\]
of ordered pairs of level sets of $h$ and the corresponding ordered
2n-tuple
\[(g(h|_N^{-1}(r_1), h|_M^{-1}(r_1)), \dots, g(h|_N^{-1}(r_n), h|_M^{-1}(r_n)))
\in {\mathbb R}^n \in
{\mathbb R}^{\infty}.\] Set
\[w_h(N) = f(g(h|_N^{-1}(r_1), h|_M^{-1}(r_1)), \dots, g(h|_N^{-1}(r_n),
h|_M^{-1}(r_n))).\] We call
$w_h(N)$ the width of $N$ relative to $h$.  Set \[w(N) = \min\{ \; w_h(N) \; | \;
h \; \in \; {\cal C}\}. \; \] We call $w(N)$ the width of $(N, C, g, {\cal O}, f)$.
We say that $(N, C, g, {\cal O}, f)$ is in thin position if it is
presented together with $h \in {\cal C}$ such that $w(N) = w_h(N)$.

If $r_i$ is such that \[g(h|_N^{-1}(r_{i-1}), h|_M^{-1}(r_{i-1})) <
g(h|_N^{-1}(r_{i}), h|_M^{-1}(r_{i})) > g(h|_N^{-1}(r_{i+1}),
h|_M^{-1}(r_{i+1}))\] where $<$ and $>$ are in the dictionary order, then we call
$(h|_N^{-1}(r_i),h|_M^{-1}(r_i))$ a thick level.  If $r_i$ is such that
\[g(h|_N^{-1}(r_{i-1}), h|_M^{-1}(r_{i-1})) > g(h|_N^{-1}(r_{i}),
h|_M^{-1}(r_{i})) < g(h|_N^{-1}(r_{i+1}), h|_M^{-1}(r_{i+1}))\] in
the dictionary order, then we call $(h|_N^{-1}(r_i),h|_M^{-1}(r_i))$ a
thin level.

\subsection{Thin position for knots}

The notion of thin position for knots was introduced by D. Gabai.  He
designed and used this notion successfully to prove Property R for
knots.  We here specify $(N, M)$, $C$, $g$, ${\cal O}$ and $f$ as used
in the context of thin position for knots.  Let \[(N, M) = (K,
{\mathbb S^3})\] be a knot type.  Take $C$ to be the requirement that
the Morse function $h\co  (K, {\mathbb S^3}) \rightarrow {\mathbb R}$
has exactly two critical points on ${\mathbb S}^3$ (a maximum, $\infty$, and a
minimum, $-\infty$); we call such a function a (standard) height function of 
$\mathbb{S}^3$.  In considering thin position for knots, we may 
visualize our Morse function as projection onto the vertical
coordinate.  The fact that we may do so derives from the constraint
placed on the Morse functions under consideration.

Let $g$ be the function that takes the ordered pair
\[(h|_{K}^{-1}(r_i), h^{-1}(r_i))\] of level sets of a Morse function
$h$ to \[\chi(h|_{K}^{-1}(r_i))\] And let ${\cal O}$ be ${\mathbb N}$
and $f\co  {\mathbb R}^{\infty} \rightarrow {\mathbb N}$ the function
defined by \[f(x_1, \dots, x_n) = \sum_i x_i\]

Thus in this case, we proceed as follows: Given a Morse function $h\co 
(K, {\mathbb S^3}) \rightarrow {\mathbb R}$ of pairs such that $h|_M$
is a height function, let $c_0, \dots, c_n$ be the critical points of
$h$.  Note that since these critical points are critical points of
either $h|_K$ or of $h|_{{\mathbb S}^3}$, exactly two of these
critical points will be critical points of $h|_{{\mathbb S}^3}$.  Note
further that one of these critical points lies below all critical
points of $h|_K$ and the other lies above all critical points of
$h|_K$.

Now, for $i = 1, \dots, n$, choose regular values $r_i$ such that
$c_{i-1} < r_i < c_i$.  Consider pairs of level surfaces
\[(h|_K^{-1}(r_i), h|_{{\mathbb S}^3}^{-1}(r_i))\]
and
\[g(h|_K^{-1}(r_i), h|_{{\mathbb S}^3}^{-1}(r_i)) = \chi(h|_K^{-1}(r_i)) =
\#|K \cap (h|_{{\mathbb S}^3})^{-1}(r_i)|\]
Note that here \[h|_K^{-1}(r_1)) = h|_K^{-1}(r_n)) = \emptyset\] and
thus \[g(h|_K^{-1}(r_1), h|_{{\mathbb S}^3}^{-1}(r_1)) =
g(h|_K^{-1}(r_n), h|_{{\mathbb S}^3}^{-1}(r_n)) = 0.\]
This yields the ordered n-tuple
\[(0, \#|K \cap (h|_{{\mathbb S}^3})^{-1}(r_2)|, \dots,
\#|K \cap (h|_{{\mathbb S}^3})^{-1}(r_{n-1})|, 0)\]
And thus
\[w_h(K, {\mathbb S}^3) = 0 + \#|K \cap (h|_{{\mathbb S}^3})^{-1}(r_2)| +
\dots + \#|K \cap (h|_{{\mathbb S}^3})^{-1}(r_{n-1})| + 0\]

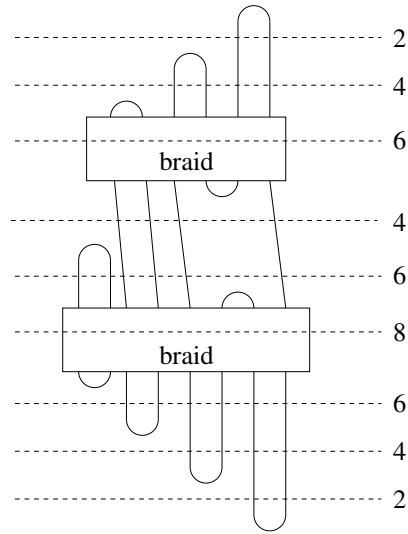
\begin{figure}[ht!]
\labellist\small
\pinlabel {2} [l] at 425 559
\pinlabel {4} [l] at 425 505
\pinlabel {6} [l] at 425 443
\pinlabel {4} [l] at 425 351
\pinlabel {6} [l] at 425 290
\pinlabel {8} [l] at 425 226
\pinlabel {6} [l] at 425 145
\pinlabel {4} [l] at 425 91
\pinlabel {2} [l] at 425 37
\pinlabel {braid} at 200 420
\pinlabel {braid} at 200 200
\endlabellist
\centerline{\includegraphics[width=.4\textwidth]{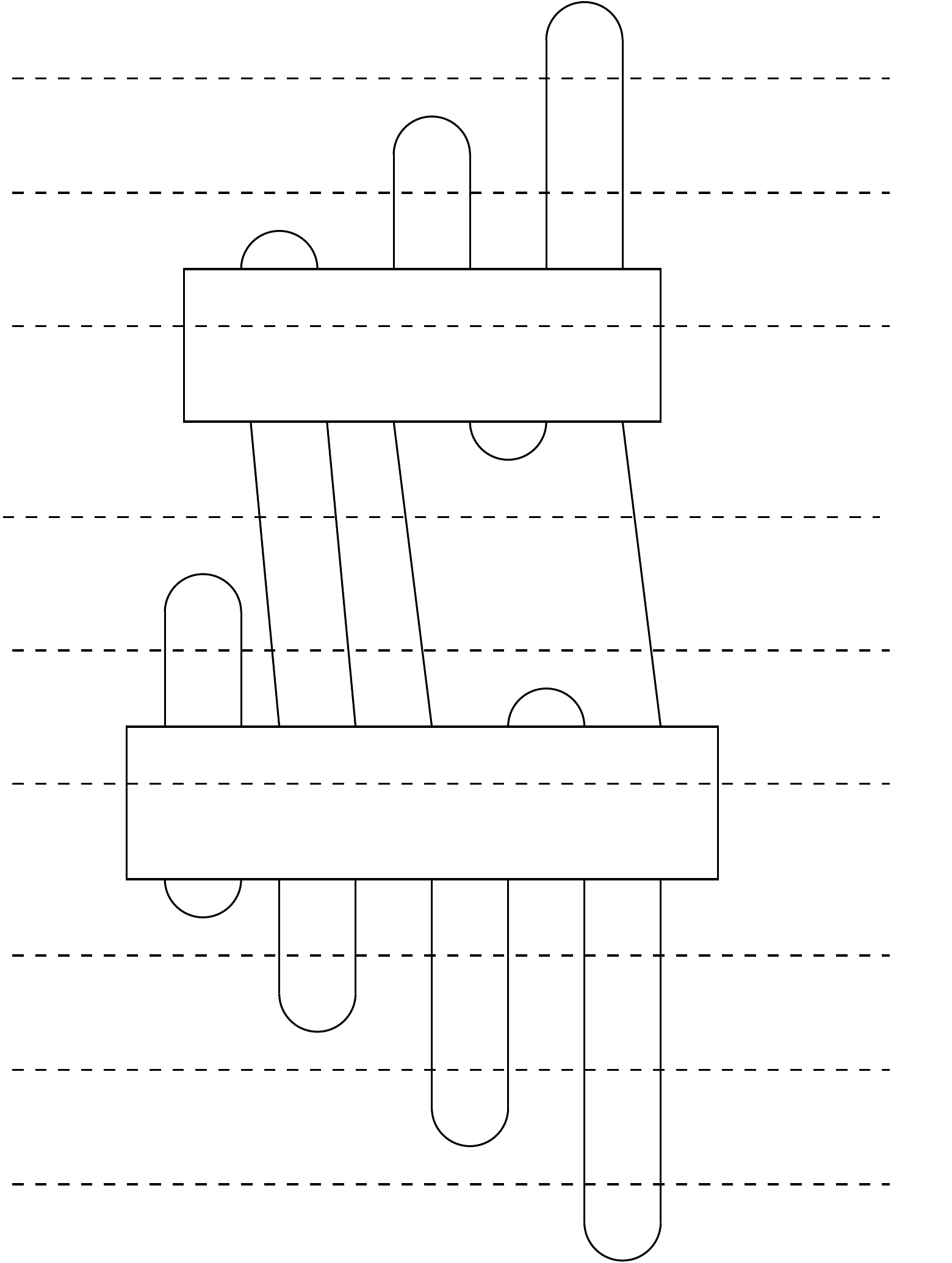}}
\caption{Thin position for knots}
\label{tpk}
\end{figure}

In \fullref{tpk}, the knot pictured schematically has
\[w_h(K, {\mathbb S}^3) = 0 + 2 + 4 + 6 + 4 + 6 + 8 + 6 + 4 + 2 + 0 = 42\]
The width of $(K, {\mathbb S}^3)$ is the smallest possible relative
width $w_h(K)$, as $h$ ranges over all height functions on ${\mathbb
S}^3$.  In the usual computation of width, one considers only critical
points of $h|_K$, one thus considers two fewer critical points and two
fewer regular points and is thus not compelled to add the $0$'s in the
sum.

\subsection{Thin position for 3--manifolds} \label{tp3m}

The notion of thin position for 3--manifolds was pioneered by
Scharlemann and Thompson.  We here specify $(N, M)$, $C$, $g$, ${\cal
O}$ and $f$ as used in the context of thin position for 3--manifolds.
Let $N = M$ and let $M$ be a closed 3--manifold.  Let $C$ be
the vacuous requirement (we consider all Morse functions).  Let $g$ be the
function that takes the ordered pair
\[(h^{-1}(r_i), \emptyset)\] of level sets of a Morse function
$h$ to \[\#|h^{-1}(r_i)| + s_i - \chi(h^{-1}(r_i)),\] where $s_i$ is the
number of $S^2$ components in $h^{-1}(r_i)$.  Let ${\cal O}$ be
${\mathbb N}^{\infty}$ in the dictionary order.  Finally, let $f\co 
{\mathbb R}^{\infty} \rightarrow {\mathbb Z}^{\infty}$ be the function
that takes the ordered $n$--tuple $(x_1, \dots, x_n)$, deletes all
entries $x_i$ for which either $x_{i-1} > x_i$ or $x_{i+1} > x_i$ and
then arranges the remaining entries (that is, the local maxima) in
nonincreasing order. 

Thus in this case, we proceed as follows: We identify $(M,M)$ with $M$.  
Let $h$ be a Morse function $$h\co  M \rightarrow {\mathbb R}.$$ 

Let $c_0, \dots, c_n$ be the critical points of $h$ and for $i = 1, 
\dots, n$, choose regular values $r_i$ such that $c_{i-1} < r_i <
c_i$.  Consider the level surfaces
\[h^{-1}(r_1), \cdots, h^{-1}(r_n)\]
and
\[g(h^{-1}(r_i)) = \#|h^{-1}(r_i)| + s_i - \chi(h^{-1}(r_i))\]
where $s_i$ is the number of spherical components of $h^{-1}(r_i)$.
This yields the ordered n-tuple
\[(\#|h^{-1}(r_1)| + s_1 - \chi(h^{-1}(r_1)), \cdots, \#|h^{-1}(r_n)| + s_n -
\chi(h^{-1}(r_n))).\]

The function $f$ picks out the values $\#|h^{-1}(r_1)| + s_1 -
\chi(h^{-1}(r_1))$ for the thick levels of $h$ and arranges them in non
increasing order. Thus
\[w_h(K, {\mathbb S}^3) = f((\#|h^{-1}(r_1)| + s_1 -
\chi(h^{-1}(r_1)), \cdots, \#|h^{-1}(r_n)| + s_n -
\chi(h^{-1}(r_n))))\]
and $w(N)$ is the smallest such sequence arising for a Morse function
$h$ on $M$, in the dictionary order.

\begin{figure}[ht!]
\labellist\small
\pinlabel {2--handles} at 125 128
\pinlabel {1--handles} at 68 100
\pinlabel {2--handles} at 68 48
\pinlabel {1--handles} at 125 20
\pinlabel {$S_2 =$ genus 2} [l] at 260 110
\pinlabel {$F_1 = 2 \times $ genus 1} [l] at 260 75
\pinlabel {$S_1 = $ genus 2} [l] at 260 40
\endlabellist
\centerline{\includegraphics[width=0.5\textwidth]{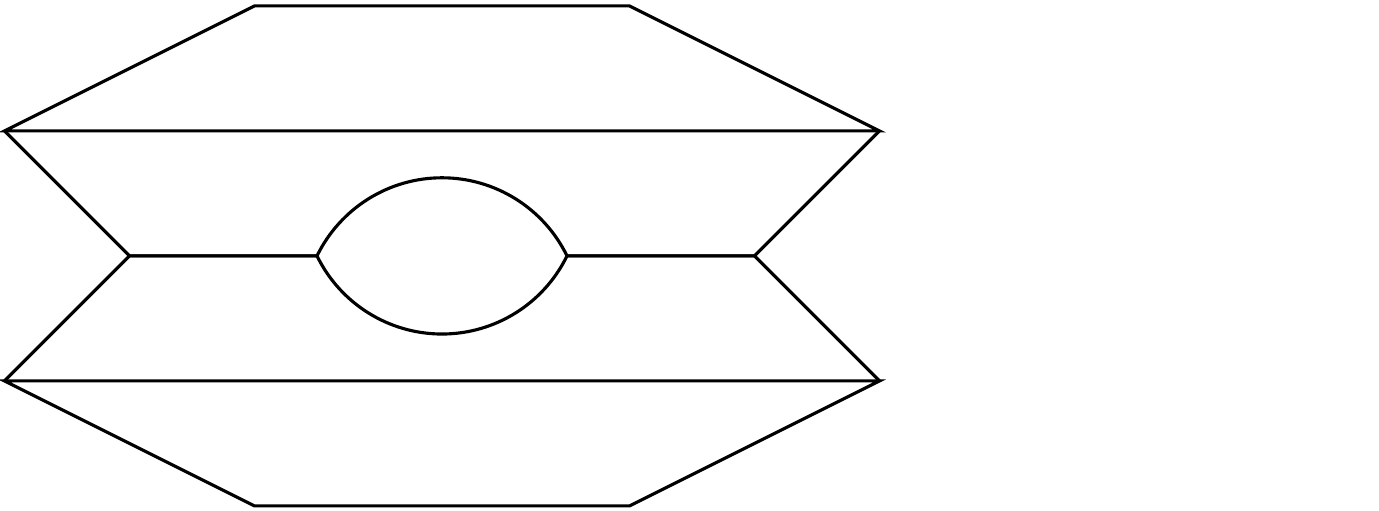}}
\caption{Thin position for 3--manifolds}
\label{tp3mf}
\end{figure}
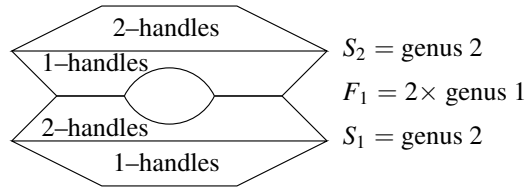

The schematic in \fullref{tp3mf} describes a decomposition of the 3--torus
\[{\mathbb T}^3 = {\mathbb S}^1 \times {\mathbb S}^1 \times {\mathbb S}^1.\]
Note that a torus or a sphere will never appear as a thick level in a thin
presentation of $\mathbb{T}^3$, and a single genus 2 surface is insufficient.
So the width of $\mathbb{T}^3$ is:
\[w_h({\mathbb T}^3) = (3,3).\]

\subsection{Thin position for knots in 3--manifolds}

We here suggest a more general application of the notion of thin
position to knots in 3--manifolds.
This notion differs from the 
standard notion of thin position for knots in ${\mathbb S}^3$ in that
we do not restrict our attention to specific height functions.  In the setting
of 3--manifolds we wish to pick Morse functions optimal with respect to
both the 3--manifold and the knot.

\begin{rem}
The first application of thin position for knots in general manifolds was
given in 1997 in two independent PhD dissertations: Feist (unpublished) and 
Rieck \cite{rieck}.  However, their approach is different from ours and is
described below.   An similar approach to the one presented here can be found
in Hayashi and Shimokawa \cite{HSh}.
\end{rem}

We here specify the $(N, M)$, $C$, $g$, ${\cal O}$ and $f$ we have in
mind.  Let $M$ be a closed 3--manifold and let $N = K$ be a knot contained
in $M$.  Let $C$ be the vacuous requirement.  Let $g$ be the function
that takes the ordered pair
\[((h|_K)^{-1}(r_i), h^{-1}(r_i))\] of level sets of a Morse function
$h$ to \[2\#|h^{-1}(r_i)| - \chi(h^{-1}(r_i)) + 2\#|h|_K^{-1}(r_i)| -
\chi(h|_K^{-1}(r_i)).\] (Here the last two terms just count the
number of points $\#|h|_K^{-1}(r_i)|$.  The cumbersome notation aims
to emphasize the equal weight of the 3--manifold and the knot.)  And
let ${\cal O}$ be ${\mathbb N}^{\infty}$ in the dictionary order.
Finally, let $f\co  {\mathbb R}^{\infty} \rightarrow {\mathbb
N}^{\infty}$ be the function that takes the ordered $n$--tuple $(x_1,
\dots, x_n)$ and rearranges the entries so they are in nonincreasing
order.

Thus in this case, we proceed as follows: Given a Morse function \[h\co 
(K, M) \rightarrow {\mathbb R},\] let $c_0, \dots, c_n$ be the
critical points of $h$.  For $i = 1, \dots, n$, choose regular values
$r_i$ such that $c_{i-1} < r_i < c_i$.  Consider the pairs
\[(h|_K^{-1}(r_1), h^{-1}(r_1)), \cdots, (h|_K^{-1}(r_n), h^{-1}(r_n))\]
then
\[g((h|_K^{-1}(r_1), h^{-1}(r_i))) = 2\#|h^{-1}(r_i)| -
\chi(h^{-1}(r_i)) + 2\#|h|_K^{-1}(r_i)| - \chi(h|_K^{-1}(r_i)).\]

This yields the ordered n-tuple
\[(2\#|h^{-1}(r_1)| - \chi(h^{-1}(r_1)) + 2\#|h|_K^{-1}(r_1)| -
\chi(h|_K^{-1}(r_1)), \cdots, \]\[2\#|h^{-1}(r_n)| - \chi(h^{-1}(r_n)) +
2\#|h|_K^{-1}(r_n)| - \chi(h|_K^{-1}(r_n)))\]

The function $f$ rearranges the entries in non increasing order.  Thus
\begin{eqnarray*}
  w_h(K, M) &=& f((2\#|h^{-1}(r_1)| - \chi(h^{-1}(r_1)) +
                              2\#|h|_K^{-1}(r_1)| - \chi(h|_K^{-1}(r_1)),
                              \cdots,\\
                  & & \ \ \ \ 2\#|h^{-1}(r_n)| - \chi(h^{-1}(r_n)) +
                              2\#|h|_K^{-1}(r_n)| - \chi(h|_K^{-1}(r_n)))) 
\end{eqnarray*}
and $w(K,M)$ is the smallest such sequence arising for Morse
functions on $(K,M)$, in the dictionary order.

\begin{figure}[ht!]
\labellist\tiny
\pinlabel {0} [l] at 580 456
\pinlabel {2} [l] at 580 438
\pinlabel {4} [l] at 580 420
\pinlabel {6} [l] at 580 402
\pinlabel {8} [l] at 580 384
\pinlabel {10} [l] at 580 366
\pinlabel {8} [l] at 580 298
\pinlabel {6} [l] at 580 236
\pinlabel {8} [l] at 580 174
\pinlabel {10} [l] at 580 100
\pinlabel {8} [l] at 580 82
\pinlabel {6} [l] at 580 64
\pinlabel {4} [l] at 580 46
\pinlabel {2} [l] at 580 28
\pinlabel {0} [l] at 580 10
\small
\pinlabel {braid} at 290 344
\pinlabel {braid} at 290 128
\endlabellist
\centerline{\includegraphics[width=.6\textwidth]{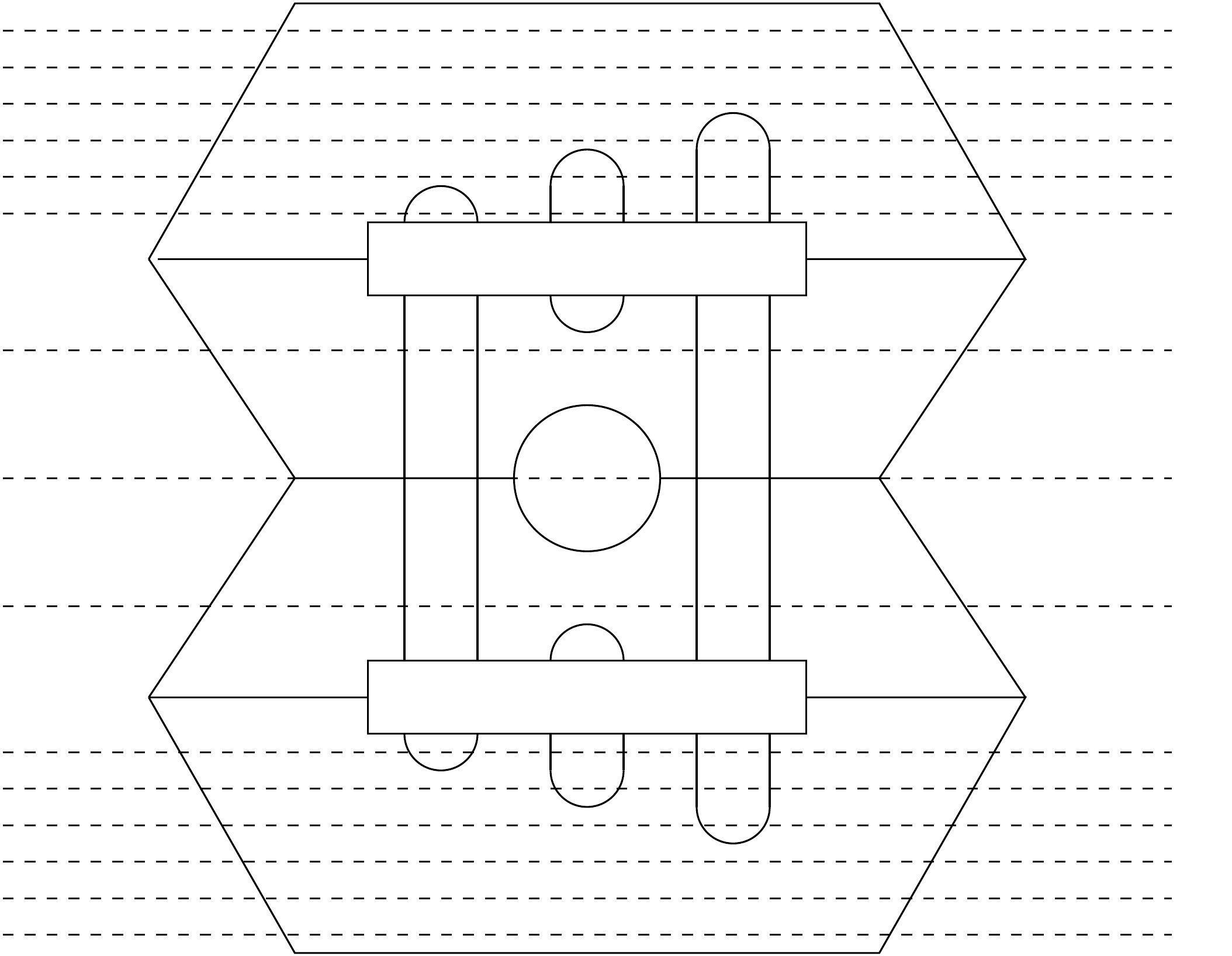}}
\caption{Thin position for knots in 3--manifolds}
\label{tpk3m}
\end{figure}
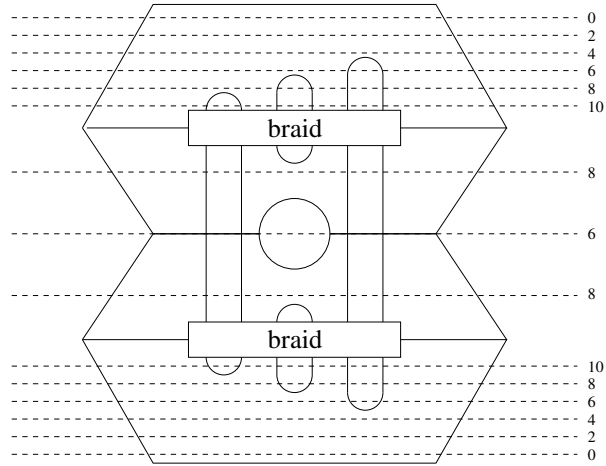

The schematic in \fullref{tpk3m} gives $g((h^{-1}(r_i),
h|_K^{-1}(r_i)))$ for a knot in ${\mathbb T}^3 = {\mathbb S}^1 \times
{\mathbb S}^1 \times {\mathbb S}^1$ with respect to a specific Morse
function.  Here

\[w_h(K, {\mathbb T}^3) = (10, 10, 8, 8, 8, 8, 6, 6, 6, 4, 4, 2, 2, 0, 0).\]

\subsection{Other settings}

There are other settings to which our general theory applies.  We will
not work them out in detail here.  One which deserves to be
mentioned is that of manifolds with boundary.  This setting has been
studied along with the case of closed 3--manifolds as in \fullref{tp3m}.  But in those studies, the functions considered are not in
fact Morse functions, but rather Morse functions relative boundary,
that is, functions that are Morse functions except that they are constant
on boundary components.

One can consider the setting in which this requirement is dropped.
Then $(N, M)$, $C$, $g$, ${\cal O}$ and $f$ are as follows:  $M$ is a
3--manifold and $N = M$ (as above we identify $(M,M)$ with $M$).  There are no
requirements on the 
Morse functions (except that they be Morse functions, in particular,
transverse to $\partial M$).  And $g$ is the function that takes the
ordered pair
\[(\emptyset,h^{-1}(r_i))\] of level sets of a Morse function
$h$ to \[\#|h^{-1}(r_i)| + s_i - \chi(h^{-1}(r_i)),\] 
where $s_i$ is the number of spheres in $h^{-1}(r_i)$ and ${\cal O}$ is
${\mathbb N}^{\infty}$ in the dictionary order.  Finally, $f\co  {\mathbb
R}^{\infty} \rightarrow {\mathbb N}^{\infty}$ is the function that
takes the ordered $n$--tuple $(x_1, \dots, x_n)$, deletes all entries
$x_i$ for which either $x_{i-1} > x_i$ or $x_{i+1} > x_i$ and then
arranges the remaining entries in nonincreasing order.  Much of the
theory of Scharlemann and Thompson should carry over to this setting.

As mentioned above, the definition of thin position for a knot $K$ in a
3--manifold $M$ given by Feist and Rieck \cite{rieck} is different than the
definition above.  It  
does not take into account critical points of the manifold.  We can retrieve
it by considering Morse functions with the following constraints: all the
critical points of $M$ of index zero or one are in $h^{-1}(-\infty,-1)$,  all
the critical points of $M$ of index two or three are in $h^{-1}(1,\infty)$,
and the knot in contained in $h^{-1}(-1,1)$.  The width is then calculated as
in $\mathbb{S}^3$ by summing the number of times $K$ intersects each level:
\[w_h(K, M) = 0 + \#|K \cap (h|_{M})^{-1}(r_2)| +
\dots + \#|K \cap (h|_{M})^{-1}(r_{n-1})| + 0\]
Another important setting is graphs embedded in 3--manifolds.  Although this
paper is about knots and 3--manifolds, we can generalize the definition of thin
position by allowing $N$ to be a graph.  A simple application of this was
given by Rieck and Sedgwick \cite{riecksedgwick} where the authors
considered a bouquet of circles (that is, a connected graph with a single
vertex).  The constraint imposed is equivalent to: all the
critical points of $M$ of index zero or one are in $h^{-1}(-\infty,-1)$,  all
the critical points of $M$ of index two or three are in $h^{-1}(1,\infty)$,
the vertex is at level $1$, and the interiors of all the edges are in
$h^{-1}(-1,1)$.  Again, the width was calculated as above.  A more
sophisticated approach was taken by Scharlemann and Thompson \cite{schthomp}
and Goda, Scharlemann and Thompson \cite{godaschthomp}, who considered
trivalent graphs (that is, graphs with vertices of valence 3 only) in
$\mathbb{S}^3$.  They used the standard height function on $\mathbb{S}^3$.
Roughly speaking, they treated a vertex as a critical point.  Generically,
every vertex has two edges above and one below (a $Y$ vertex) or two edges
below and one above (a $\lambda$ vertex).  The treatment of $Y$ vertices is
similar to that of minima and of $\lambda$ vertices to that of maxima.

\section{A counting argument (or why forgetfulness is practically
irrelevant)}

In this section we discuss a counting argument that relates two
different widths if these widths are computed identically except at
the final stage.  That is, if $(N, M)$, $C$ and $g$ are identical, but
${\cal O}$ and $f$ differ in a prescribed way.

As a warm up, consider the lemma below.  It is based on a comment by
Clint McCrory.  We say that a knot $K$ in ${\mathbb S}^3$ is in bridge
position with respect to the height function $h$, if all its maxima
occur above all its minima.  The bridge number of $K$ is the smallest
possible number of maxima as $h$ ranges over all height functions on
${\mathbb S}^3$ (see Schultens \cite{Sc3}).  In
\fullref{section:additivity} we give a more detailed discussion of
bridge position and its relationship to thin position.

\begin{lem}[Clint McCrory] \label{bridgecount}
Let $K$ be a knot in ${\mathbb S}^3$.  If thin
position is necessarily bridge position and the bridge number of $K$
is $n$, then $w(K) = 2n^2$.
\end{lem}

\begin{proof}
Suppose the knot is in thin position with respect to $h$ and is also
in bridge position.  Then the knot has $k$ maxima and $k$ minima, for
$k \geq n$.  If we denote the critical values in increasing order by
$c_0, \dots, c_{k-1}, c_k, \dots, c_{2k}$, then $c_0, \dots, c_{k-1}$ are
minima and $c_k, \dots, c_{2k}$ are maxima.  Thus \[h|_{K}^{-1}(r_1) = 2,
h|_{K}^{-1}(r_2) = 4, \cdots, h|_{K}^{-1}(r_{k}) = 2k\]
\[h|_{K}^{-1}(r_{k+1}) = 2k - 2, 
h|_{K}^{-1}(r_{k+2}) = 2k - 4, \cdots, h|_{K}^{-1}(r_{2k}) = 2.\] 
See \fullref{calc}.  There each dot corresponds to
$\frac{h|_{K}^{-1}(r_i)}{2}$ in the case where $k = 5$.

\begin{figure}[ht!]
\centerline{\includegraphics[width=.3\textwidth]{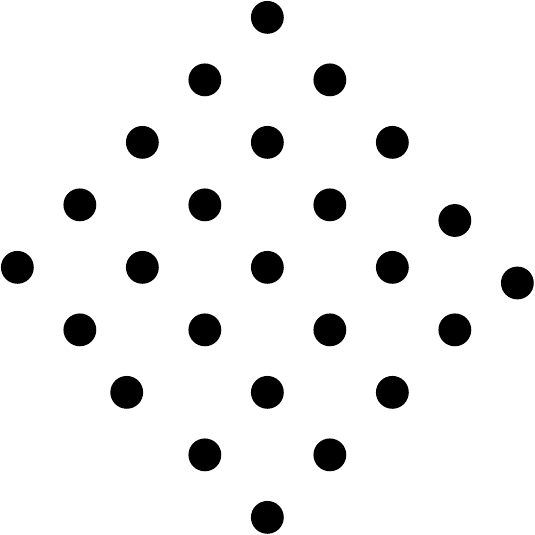}}
\caption{Calculating width}
\label{calc}
\end{figure}

Note how the total number of dots is $k^2$.  (This is merely a
geometric visualization of the Gauss summation formula.)  Thus $w_h(K)
= 2k^2$.  Now since bridge number is $n$, we see that $h$ can be
chosen so that $w_h(K) = 2n^2$.  Since thin position is necessarily
bridge position, it follows that $w(K) = 2n^2$.
\end{proof}

A slightly more general version of this lemma allows us to compute the
width of a knot from the thick and thin levels of a knot in thin
position.  This more general lemma was included in Scharlemann and
Schultens \cite{SS}.

\begin{lem} \label{widthcount}
Let $S_{i_1}, \dots, S_{i_k}$ be the thick levels of $K$ and $F_{j_1},
\dots, F_{j_{k-1}}$ the thin levels.  Set $a_{i_l} = \frac{|\;K\; \cap
S_{i_l}|}{2}$ and $b_{j_l} = \frac{|\;K\; \cap F_{i_l}|}{2}$.  Then \[w(K) = 2
\sum_{l = 1}^k a_{i_l}^2 - 2 \sum_{l = 1}^{k-1} b_{j_l}^2.\]
\end{lem}

\begin{proof}
We prove this by repeated use of the Gauss Summation Formula.
In particular, we use the Gauss summation formula on the squares
arising from thick levels.  Then note that when we do so, we count the
small squares arising from the thin levels twice.  To compensate, we
subtract the appropriate sums.  See \fullref{canc}.
\end{proof}

\begin{figure}[ht!]
\centerline{\includegraphics[width=.3\textwidth]{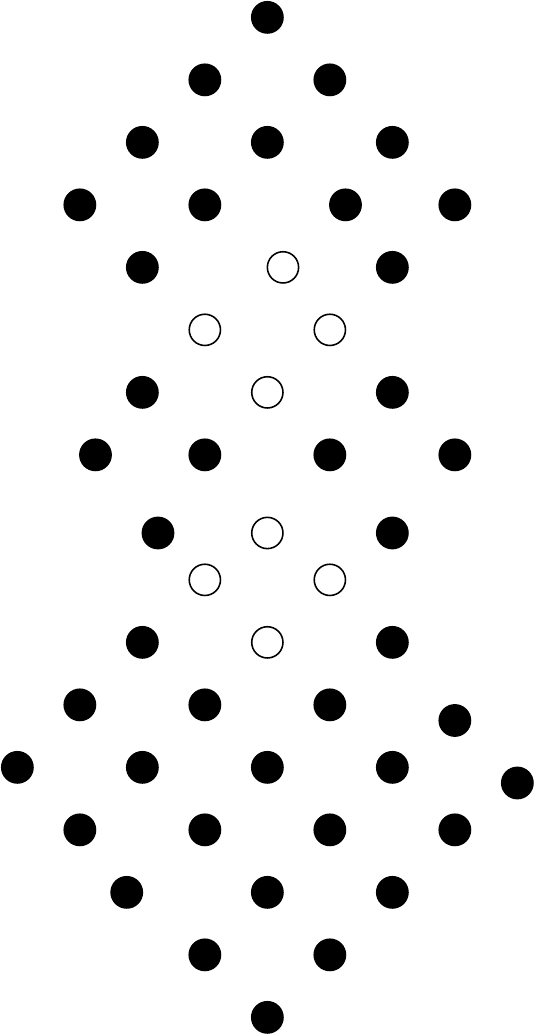}}
\caption{A cancellation principle}
\label{canc}
\end{figure}

One consequence of this Lemma is the following: When defining thin
position for knots in the traditional way as above, the relevant
information is captured in the thick and thin levels.  An alternate
definition would thus be to use ${\mathbb N}^{\infty}$ instead of
${\mathbb Z}$ for ${\cal O}$ and to let $f$ be the function that picks
out $g(h|_K^{-1}(r_i))$ for the thick and thin levels.  This would be
slightly more informative than the traditional definition.  Then, if
$f$ also rearranges the remaining entries in non increasing order, we
lose information.  In the applications of thin position of knots to
the study of 3--manifolds these subtleties in the definitions appear to
be irrelevant.

\section{Key features of thin position}

The notion of thin position was introduced by D. Gabai with a specific
purpose in mind.  It provided a way of describing a positioning of
knots in $\mathbb{S}^3$ that made certain arguments about
surfaces in the knot exterior possible.  The key feature of thin position for
a knot lies in the absence of disjoint pairs of upper and lower disks with
respect to a regular value $r$ of a Morse function: An upper (lower)
disk for a knot $K$ with respect to the regular level $r$ of a Morse
function $h$ is a disk $D$ whose boundary decomposes into two arcs,
$\alpha$ and $\beta$, such that $\alpha \in K$, $\beta \in h^{-1}(r)$
and such that $h(a) > h(r)$ ($h(a) < h(r)$) for all $a$ in the
interior of $\alpha$.  We emphasize that parts of the interior of a upper
(lower) disk may be below (above) $h^{-1}(r)$.

Now suppose that $K$ is in thin position with
respect to the Morse function $h$.  Further suppose that $D$ is an
upper disk for $K$ with respect to $r$ and $E$ is a lower disk for $K$
with respect to $r$.  If $D \cap E = \emptyset$, then we may isotope
the portion of $K$ in $\partial D$ just below $h^{-1}(r)$ and the
portion of $K$ in $\partial E$ just above $h^{-1}(r)$ to obtain a
presentation of $K$ that intersects $h^{-1}(r)$ four fewer times.  See
\fullref{disks} and \fullref{isotopy}.  It follows that after this
isotopy the width is reduced by exactly four if $K$ has exactly one maximum on 
$\partial D$ above $h^{-1}(r)$ and exactly one minimum on
$\partial E$ below $h^{-1}(r)$; if $K$ has more critical points on $\partial
D$ above $h^{-1}(r)$ or $\partial E$ below $h^{-1}(r)$ the width is reduced by
more than four.  (Note that if $D$ dips below $h^{-1}(r)$ or $E$ above it,
during the isotopy the width may increase, temporarily.)

\begin{figure}[ht!]
\labellist\small
\pinlabel {$D$} [br] at 125 95
\pinlabel {$E$} [tl] at 350 30
\pinlabel {$h^{-1}(r)$} [tl] at 465 65
\endlabellist
\centerline{\includegraphics[width=.7\textwidth]{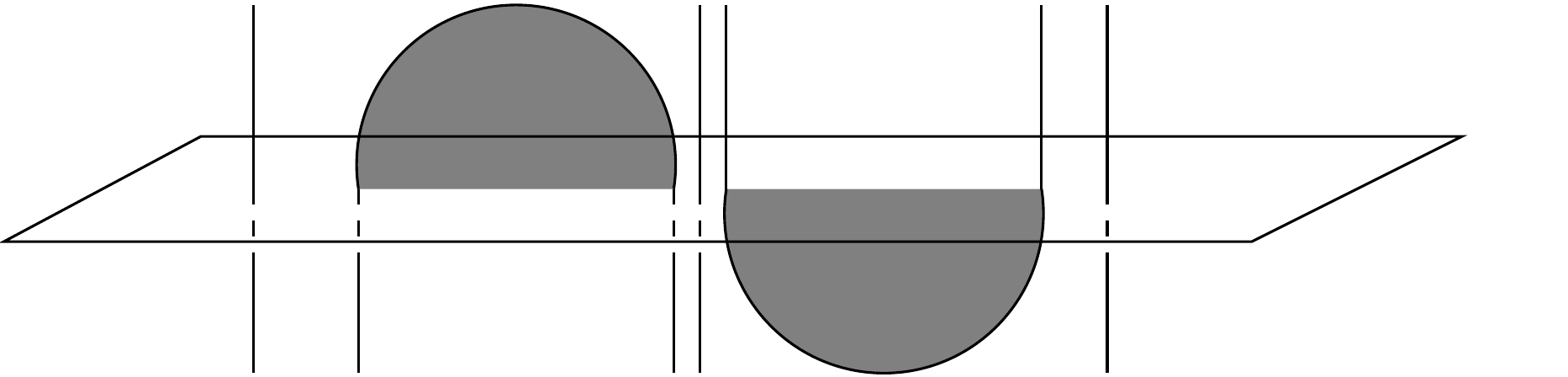}}
\caption{Two disks describing an isotopy}
\label{disks}
\end{figure}

\begin{figure}[ht!]
\labellist\small
\pinlabel {$h^{-1}(r)$} [tl] at 465 65
\endlabellist
\centerline{\includegraphics[width=.7\textwidth]{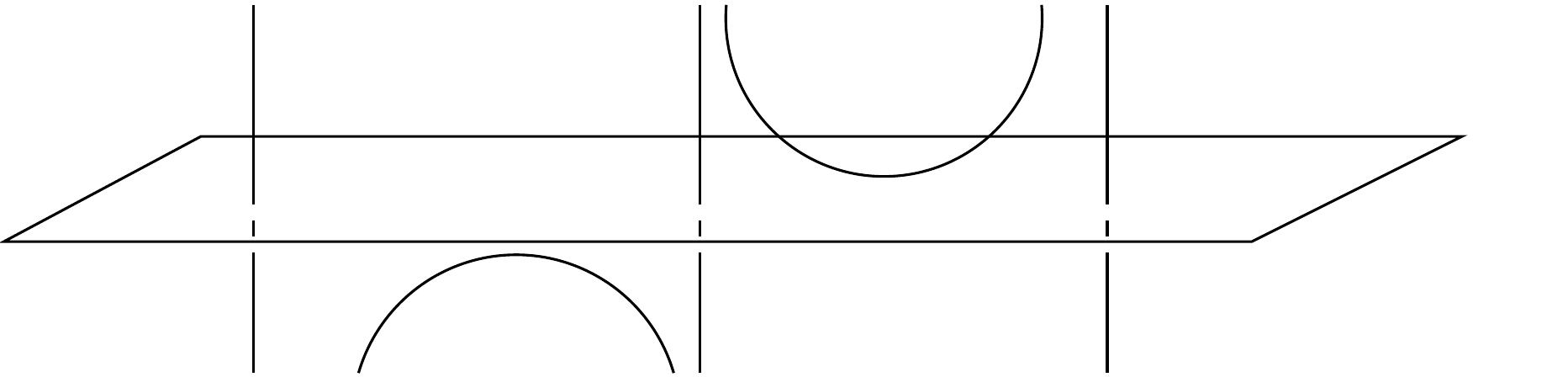}}
\caption{After the isotopy}
\label{isotopy}
\end{figure}

To make sense out of this isotopy from the point of view of thin
position, note that we may instead keep $K$ fixed and alter $h$ in
accordance with the isotopy.  We obtain a new Morse function $h'$ that
coincides with $h$ outside of a neighborhood of $D \cup E$ and such that
\[w_{h'}(K) \leq w_h(K) - 4.\]
But this contradicts the fact that $K$ is in thin position with
respect to $h$.

The situation is similar if $D \cap E$ consists of
one point.  There the relative width can be reduced by a count of
2 or more. See \fullref{disks'} and \fullref{isotopy'}.

\begin{figure}[ht!]
\labellist\small
\pinlabel {$D$} [br] at 125 95
\pinlabel {$E$} [tl] at 340 30
\pinlabel {$h^{-1}(r)$} [tl] at 465 65
\endlabellist
\centerline{\includegraphics[width=.7\textwidth]{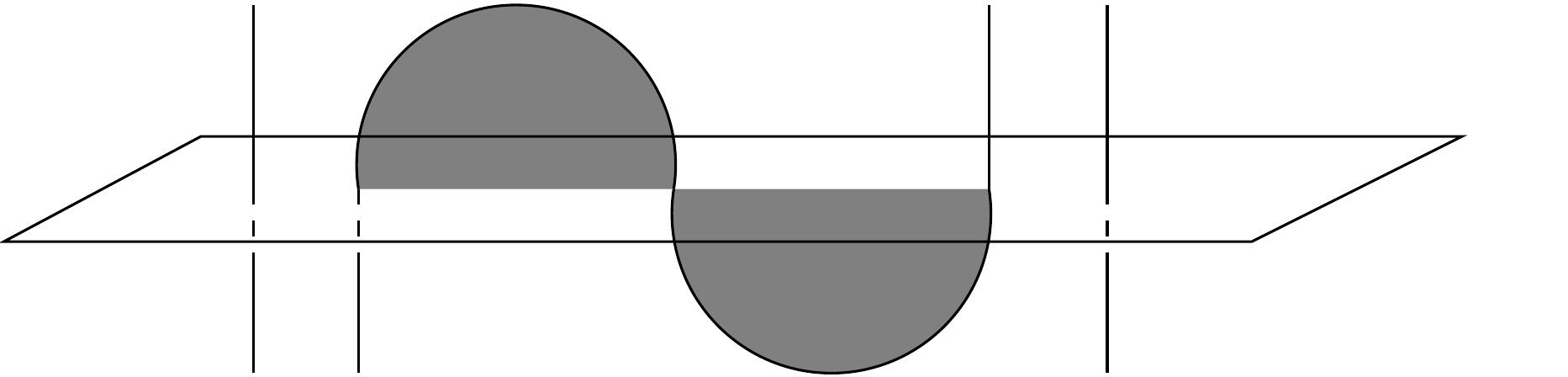}}
\caption{Two disks describing an isotopy}
\label{disks'}
\end{figure}

\begin{figure}[ht!]
\labellist\small
\pinlabel {$h^{-1}(r)$} [tl] at 465 65
\endlabellist
\centerline{\includegraphics[width=.7\textwidth]{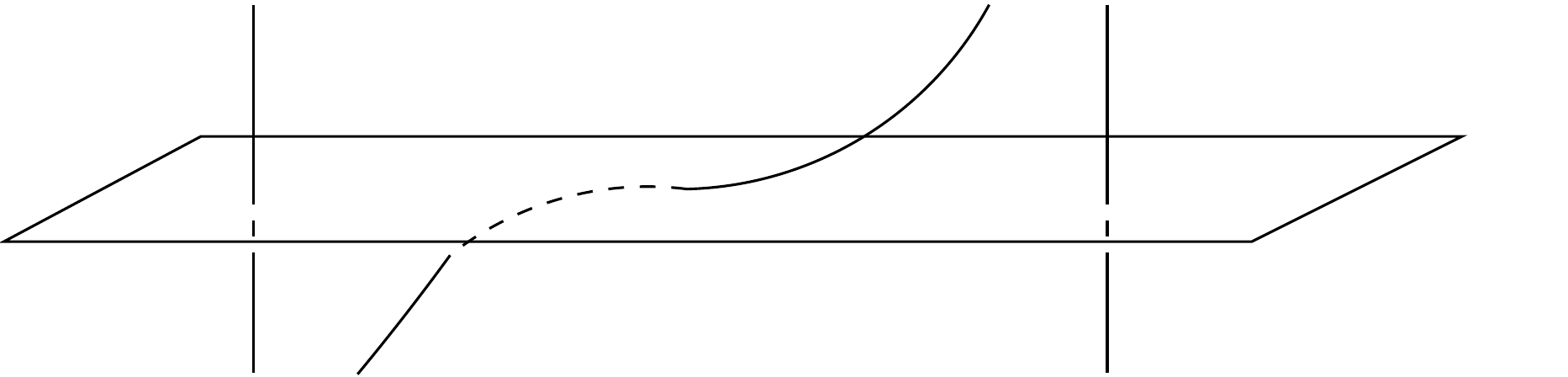}}
\caption{After the isotopy}
\label{isotopy'}
\end{figure}

Finally, consider the case in which $D \cap E$ consists of two points.
Then one subarc of $K$ lies in $\partial D$, another in $\partial E$
and the two meet in their endpoints.  It follows that $K$ can be
isotoped into the level surface $h^{-1}(r)$.  In the context of knots
in ${\mathbb S}^3$, $h^{-1}(r)$ is a 2--sphere and it then
follows that $K$ is trivial.

In the applications of thin position for knots to problems in
3--manifold topology the key feature used is the absence of disjoint
upper and lower disks with respect to a regular value.  This property
is termed {\it locally thin\/} by D\,J Heath and T Kobayashi who investigate
this property in \cite{HK2}.

When M Scharlemann and A Thompson introduced their notion of thin
position for 3--manifolds in \cite{ST}, they established a number of
properties enjoyed by a 3--manifold in thin position.  Let $M$ be a
3--manifold in thin position with respect to the Morse function $h$.
An upper (lower) compressing disk with respect to the regular value
$r$ is a disk whose boundary is an essential curve in $h^{-1}(r)$ 
whose interior, near $\partial D$ lies above (below) $h^{-1}(r)$; we further
impose that $\mbox{int}D \cap h^{-1}(r)$ consists entirely of curves that are
inessential in $h^{-1}(r)$.  This is analogous to an upper (lower) disk
dipping below (above) the level $h^{-1}(r)$.  Note that since the curves of
$\mbox{int}D \cap h^{-1}(r)$ are inessential in $h^{-1}(r)$, an upper (lower)
disk $D$ may be isotoped (relative to the boundary) to lie entirely above
(below) $h^{-1}(r)$.  However, this flexibility built into this somewhat
cumbersome definition is necessary for some applications.

This gives an analogy
with the situation for knots in ${\mathbb S}^3$: If there are upper
and lower disks with respect to $r$, then their boundaries must
intersect.

In fact, 3--manifolds in thin position enjoy a broader spectrum of
properties.  Some of these can be phrased in the language of Heegaard
splittings.  A {\em compression body} $W$ is a 3--manifold obtained
from a closed (and possibly empty) surface $\partial_- W$ by taking
$\partial_- W \times I$ (and, perhaps, some balls) and attaching
1--handles along $\partial_- W \times \{ 1 \} \subset \partial_- W
\times I$, where $I = [0,1]$, and the boundaries of the balls.  Then
$\partial_- W$ is identified with $\partial_- W \times \{0\}$ and
$\partial W \setminus \partial_- W$ is denoted $\partial_+ W$.
Dually, a compression body is obtained from a connected surface
$\partial_+ W$ by attaching 2--handles to $\partial_+ W \times \{ 0 \}
\subset \partial_+ W \times I$ and 3--handles to any resulting
2--spheres.  A Heegaard splitting of a closed 3--manifold $M$ is a
decomposition, $M = V \cup_S W$, into two handlebodies, $V, W$, such
that $S = \partial_+V = \partial_+W$.  A Heegaard splitting $M = V
\cup_S W$ is strongly irreducible if for any disk $(D, \partial D)
\subset (V, \partial_+V)$ with $\partial D$ essential in $\partial_+V$
and disk $(E, \partial E) \subset (W, \partial_+W)$ with $\partial E$
essential in $\partial_+W$, $E \cap D = \partial D \cap \partial E
\neq \emptyset$.  A surface $F$ in a 3--manifold $M$ is incompressible
if there is no disk in $M$ with boundary an essential curve on $F$ and
interior disjoint from $F$.

Some key properties that follow from those established by
M Scharlemann and A Thompson in \cite{ST} for a 3--manifold in thin
position are the following:
\begin{enumerate}
\item Every thin level is incompressible.
\item The thin levels cut the 3--manifold into (not necessarily connected)
submanifolds.
\item Each such submanifold contains one thick level.
\item The thick level defines a strongly irreducible Heegaard splitting
on the submanifold.
\end{enumerate}

\section{A digression: Strongly irreducible generalized Heegaard splittings}

Strongly irreducible generalized Heegaard splittings deserve to be
mentioned in this context.  A strongly irreducible generalized
Heegaard splitting of a 3--manifold is a sequence of disjoint surfaces
$S_1, F_1, \dots, F_{k-1}, S_k$ that has the following properties:
\begin{enumerate}
\item $S_1$ bounds a handlebody or cuts off a compression body;
\item $S_i$ and $F_i$ cobound a compression body and $S_i$ corresponds
to $\partial_+$;
\item $F_i$ and $S_{i+1}$ cobound a compression body and $S_{i+1}$
corresponds to $\partial_+$;
\item $S_k$ bounds a handlebody or cuts off a compression body;
\item the interiors of the aforementioned handlebodies
and compression bodies are disjoint;
\item $F_i$ is incompressible and $S_i$ is weakly incompressible.
\end{enumerate}
(Note that $S_i$ are strongly irreducible Heegaard surfaces for the
components of $M$ cut open along $\bigcup_{i=1}^k F_i$.)  We emphasize that
the surface above may be disconnected.

A Heegaard splitting corresponds to a handle decomposition which
corresponds to a Morse function.  Changing the order in which handles
are attached changes this Morse function by interchanging the levels
of the critical points, called a {\it handle slide}.  Given a 3--manifold $M$
and a Morse function 
$f$ corresponding to a Heegaard splitting $M = V \cup_S W$ we may
consider all Morse functions on $M$ that differ from $f$ only by
handle slides.  From the point of
view here, this gives us a condition $C$ that we impose on our Morse
functions.  Combining $C$ with $g$, ${\cal O}$ and $f$ as in the
definition of thin position for 3--manifolds yields a conditional
version of thin position for 3--manifolds.  A manifold decomposition
that is thin in this conditional sense is called an
\emph{untelescoping} of $M = V \cup_S W$.  More specifically, the
untelescoping, denoted by $S_1, F_1, \dots, F_{k-1}, S_k$, is obtained
by labeling the thick level surfaces by $S_i$ and the thin level
surfaces by $F_i$.  The results in \cite{ST} still apply in this
situation.  It follows that $S_1, F_1, \dots, F_{k-1}, S_k$ is a
strongly irreducible generalized Heegaard splitting.

The idea of untelescoping has been used very successfully in the study
of Heegaard splittings and topics related to Heegaard splitting.  We
have the following meta-theorem:

\begin{meta}
If a property holds for strongly irreducible Heegaard splittings, then
a related property holds for for all Heegaard splittings.
\end{meta}

Here is the idea behind this.  Suppose you want to prove a certain property,
let's call it $X$, and suppose that you can prove it for strongly
irreducible Heegaard splittings.  Then you should be able to prove $X$ for
essential surfaces, as these are much better behaved.  Now here's what you do: you
prove $X$ for $F_i$, since they are essential.  Then for $S_i$, since they are
strongly irreducible Heegaard splitting (although beware---they are
splittings of manifolds with boundary!).  Finally you retrieve the original
Heegaard splitting via a very well understood a process called amalgamation
\cite{Sch6}.  Now all that's left is to ask: {\it what is the related property
  that survives this ordeal?}

One of the first explicit applications of this meta-theorem can be seen
in the following two theorems concerning tunnel numbers of knots.  The
tunnel number of a knot is the least number of disjoint arcs that must
be drilled out of a knot complement to obtain a handlebody.  A
collection of such arcs is called a tunnel systems of the knot.
Tunnel systems of knots correspond to Heegaard splittings.  

A concept that deserves to be mentioned here is the following: A knot
is \emph{small} if its complement contains no closed essential
surfaces.  It follows that if a knot is small, then any tunnel system
realizing the tunnel number of the knot corresponds to a strongly
irreducible Heegaard splitting.

\begin{thm}[Morimoto--Schultens \cite{morisch}]
If $K_1, K_2$ are small knots, then \[t(K_1 \# K_2) \geq t(K_1) + t(K_2)\]
\end{thm}

It is easy to see that $t(K_1 \#K_2) \leq t(K_1) + t(K_2) +1$, so this result
is quite tight.  The meta-theorem was also used to bound below the
degeneration of tunnel number for knots that are not necessarily small:

\begin{thm}[Scharlemann--Schultens \cite{SchSch}]
\[t(K_1 \# K_2) \geq \tfrac{2}{5}(t(K_1) + t(K_2))\]
\end{thm}

We now describe another application of our meta-theorem.   A knot is called
\emph{m-small} if the meridian does not bound an essential surface; by
Culler, Gordon, Luecke and Shalen \cite[Lemma~2.0.3]{CGLS} all small knots in $\mathbb{S}^3$ are m-small.
However, minimal tunnel systems of m-small knots do not always correspond to
strongly irreducible Heegaard splittings.  Let
$K$ be a knot and $t$ its tunnel number.  Denote the bridge number of a knot
$K$ with respect to a genus $t$  Heegaard splitting by $b_1(K)$.  Morimoto
observed that if $b_1(K_1) = 1$  (such knots are also called $(t,1)$ knots)
then the tunnel number degenerate:
$t(K_1 \# K_2) < t(K_1) + t(K_2) + 1$.  He  conjectured that this is a
necessary and sufficient condition and proved this conjecture for m-small
knots in $\mathbb{S}^3$ \cite{morimoto-annalen}.  This was generalized by
applying the meta-theorem:

\begin{thm}[Kobayashi--Rieck \cite{kr-cag}]
Let $K_1 \subset M_1,\dots,K_n \subset M_n$ be m-small knots.

Then $t(\#_{i=1}^n K_1) < \Sigma_{i=1}^n t(K_i) + n-1$ if and only if there
exists a non-empty proper subset $I \subset \{1,\dots,n\}$ so that $b_1(\#_{i \in
  I} K_i) = 1$. 
\end{thm}

Oddly enough, the exact same meta-theorem that led to the generalization of
Morimoto's Conjecture for m-small knots, also led to disproving it:

\begin{thm}[Kobayashi--Rieck \cite{kr1,kr2}]
There exist knots $K_1, \ K_2 \subset \mathbb{S}^3$ so that $b_1(K_1) > 1$ and
$b_1(K_2) > 1$ but:
\[t(K_1 \# K_2) \leq t(K_1) + t(K_2).\]
\end{thm}

\section{Additivity properties}
\label{section:additivity}

Widths of knots behave erratically under connected sum of knots.
Progress in understanding this phenomenon is obstructed by the fact
that little is known about the width of specific knots.  A. Thompson
was one of the first to investigate knots in thin position in their
own right.  A knot is called \emph{meridionally planar small} (or mp-small)
if the meridian does not bound an essential meridional surface.  By
definition m-small knots are mp-small.  As mentioned
above, by the highly technical \cite[Lemma~2.0.3]{CGLS}, small knots are
m-small; thus, the family of mp-small knots contains all small knots.

Thompson proved the following theorem, see \cite{T}:

\begin{thm}[Thompson]
\label{abby}
If $K \subset S^3$ is mp-small
then a height function $h$ realizing the width
of $K$ has no thin levels.  
\end{thm}

The idea of the proof is: a thin level would give a meridional planar surface.
Compressing this surface yields an incompressible meridional planar surface.
Some amount of work then shows that this incompressible meridional planar
surface has an essential component (that is, a component that is not a
boundary parallel annulus).   In \fullref{sec:compressibility} we discuss
generalizations of this theorem.

Thus an mp-small knot $K$ in thin position has some number (say $m$) of maxima
and $m$ minima and all the maxima are above the minima.  By
\fullref{bridgecount} the width of $K$ is exactly
\[w(K) = 2m^2.\]
Let $b$ be the bridge number of $K$, that is.  Clearly, $m \geq b$.  On the
other hand, after placing $K$ in 
bridge position its width is $2b^2$, showing that $b \geq m$.  We conclude
that $m$ is the bridge number.  This is summarized in the following
corollary which is sometimes referred to informally by saying that for
mp-small knots ``thin position = bridge position'':

\begin{cor}[Thompson]
If the knot $K$ in ${\mathbb S}^3$ is mp-small, then thin
position for $K$ is bridge position.
\end{cor}

The greatest challenge in this and many other investigations of thin
position for knots is that thin levels need not be incompressible.
This fact is used to advantage by D Heath and T Kobayashi in
\cite{HK1} to produce a canonical tangle decomposition of a knot and
in \cite{HK3} to produce a method to search for thin presentations of a
knot.  M Tomova has made strides in understanding this phenomenon,
see \cite{To}.  We discuss these theories below.  In \cite{HK1},
D Heath and T Kobayashi also exhibit a knot containing a meridional
incompressible surface that is not realized as a thin level in a thin
presentation of the knot.  This propounds the idea that a decomposing
sphere for a connected sum need not be realized as a thin level in a
thin presentation of a composite knot.

One thing we do know concerning additivity properties of width of
knots is the following:
\[w(K_1 \# K_2) \leq w(K_1) + w(K_2) -2\]
To see this, stack a copy of $K_1$ in thin position on top of a copy
of $K_2$ in thin position.  The width of the connected sum is then
bounded above by the relative width of the resulting presentation.

\begin{figure}[ht!]
\labellist\small
\pinlabel {$K_1$} at 72 236
\pinlabel {$K_2$} at 72 56
\endlabellist
\centerline{\includegraphics[width=.1\textwidth]{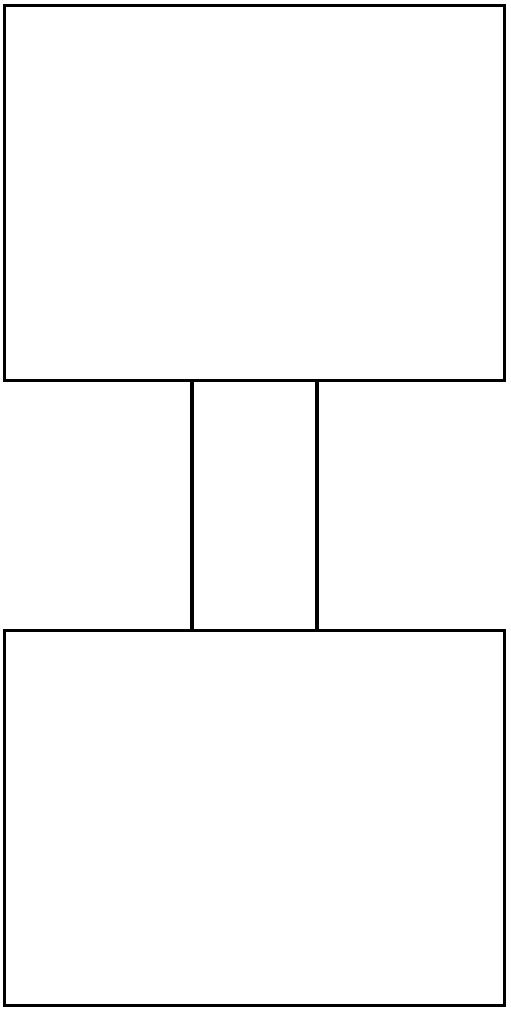}}
\caption{Connected sum of knots}
\label{sum}
\end{figure}

A result of Y Rieck and E Sedgwick proven in \cite{RS} can be
paraphrased as follows:

\begin{thm}[Rieck--Sedgwick] 
If $K_1, K_2$ are mp-small knots, then thin position of
 $K_1 \# K_2$ is related to thin position of $K_1, K_2$ as pictured in
 \fullref{sum}.  In particular, \[w(K_1 \# K_2) = w(K_1) + w(K_2)
 -2\]
\end{thm}

Given a presentation of $K_1 \# K_2$ in thin position and a
decomposing sphere $S$, Y. Rieck and E. Sedgwick proceed as follows:
They first show that the connected sum must have a thin level.  This
is accomplished as follows: For any knot $K$, \fullref{bridgecount}
gives $w(K) \leq 2b(K)^2$, where $b(K)$ is the bridge number of $K$. A
result of Schubert \cite{S}, states that the bridge number of knots is
subadditive, that is,
\[b(K_1 \# K_2) = b(K_1) + b(K_2) - 1.\]
A standard computation shows that the function
\[f(x, y) = xy - x - y + 1\]
is strictly greater than $0$ for $x, y \geq 2$.  Thus since bridge
number is always at least $2$,
\[w(K_1) + w(K_2) - 2 \leq 2b(K_1)^2 + 2b(K_2)^2 - 2 < 2(b(K_1) +
b(K_2) - 1)^2.\]
Hence thin position can't be bridge position for $K_1 \# K_2$,
there must be a thin level.

Their next steps are more technical: They show that for any
decomposing annulus in the complement of $K_1 \# K_2$, a spanning
arc can be isotoped into a thin level.  Finally, they show that
a thin level containing the spanning arc of a decomposing annulus must
in fact be a decomposing annulus.  This establishes their result.

Note that the application of Schubert's Theorem above shows that for any knot
$K_1$ and $K_2$, after placing $K_1 \# K_2$ in thin position a bridge
position is not obtained, in the sense that there is a thin sphere.  However,
counting the number of maxima in \fullref{sum} shows that:

\begin{cor}[Rieck--Sedgwick]
If $K_1$ and $K_2$ are mp-small knots,  then the number of maxima for $K_1 \#
K_2$ in thin position is the bridge number of $K_1 \# K_2$.  
\end{cor}

Examples of Scharlemann and Thompson \cite{ST1} suggest that this 
is not always the case.

The proofs in Schubert \cite{S}, Schultens \cite{Sc3} and Rieck--Sedgwick
\cite{RS} do not carry over to knots in general.  To give some idea of
the complexity of the situation, we illustrate the problems with the
strategy in \cite{Sc3}.  Rather than working with decomposing spheres
and annuli, that strategy employs swallow-follow tori.  See \fullref{sf}.

\begin{figure}[ht!]
\centerline{\includegraphics[width=.3\textwidth]{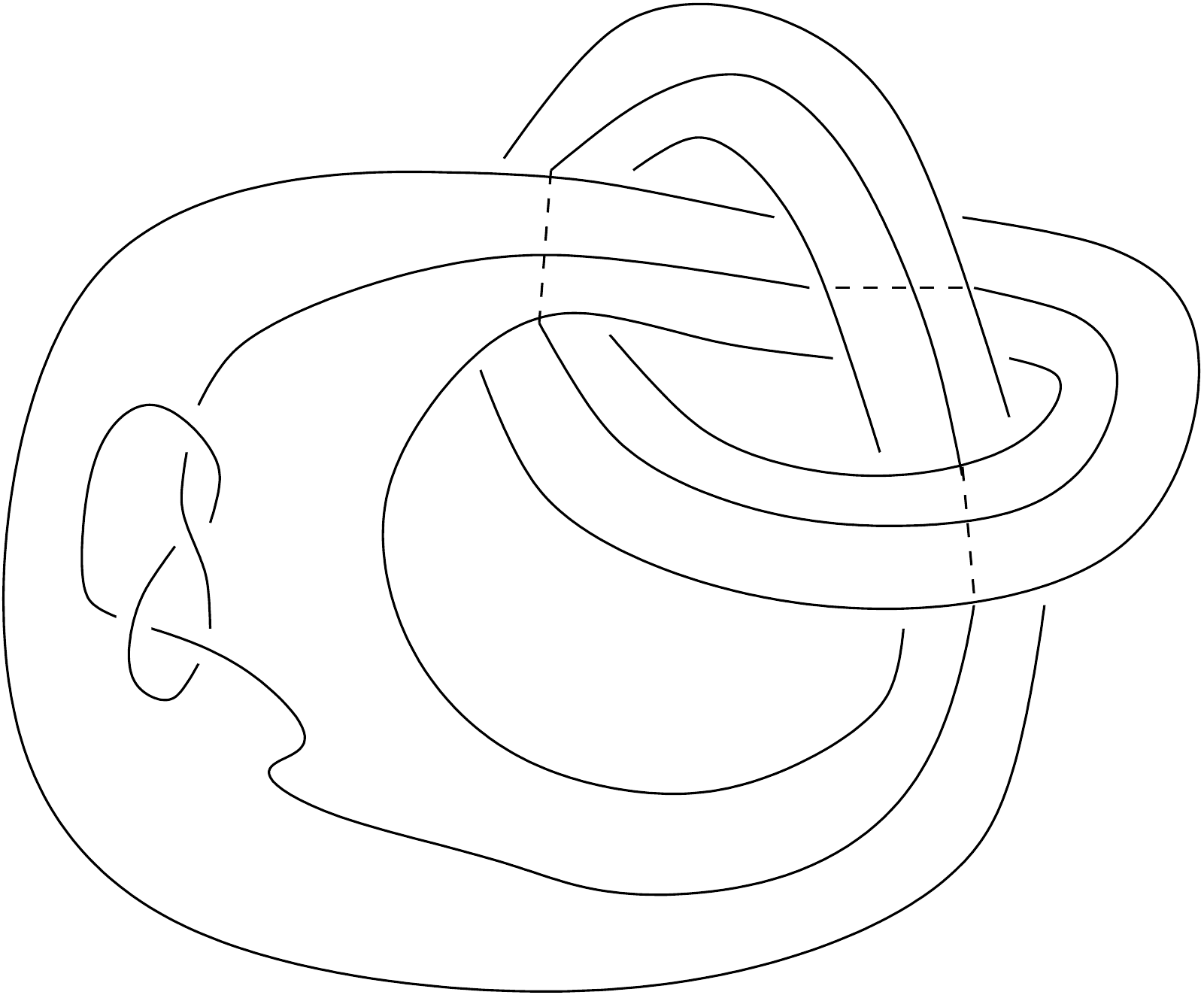}}
\caption{A swallow-follow torus}
\label{sf}
\end{figure}

Given $K_1 \# K_2$ and a decomposing sphere $S$, consider a collar
neighborhood of $(K_1 \# K_2) \cup S$ in ${\mathbb S}^3$.  Its
boundary consists of two tori.  A torus isotopic to either of these
tori is called a swallow-follow torus.  \fullref{sf} illustrates
the case in which $K_1$ is a figure eight knot and $K_2$ is a trefoil.

Swallow-follow tori often prove more effective in studying connected
sums of knots, in large part because they are closed surfaces.
The argument in \cite{Sc3} fails in settings where the swallow-follow
torus is too convoluted.  See \fullref{convoluted}.

\begin{figure}[ht!]
\centerline{\includegraphics[width=.4\textwidth]{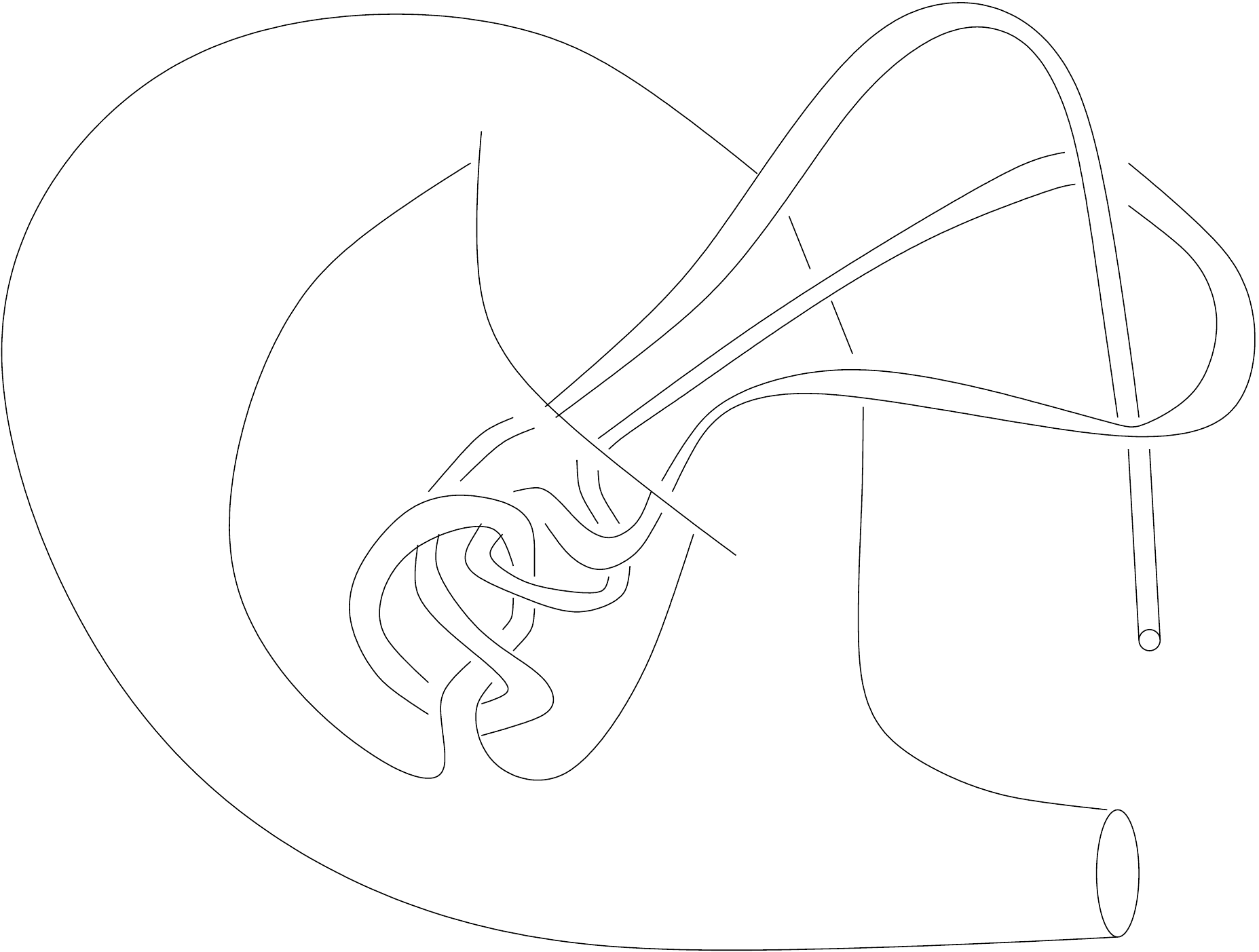}}
\caption{Convoluted portion of a swallow follow torus}
\label{convoluted}
\end{figure}

The philosophical correspondence between tunnel numbers and strongly
irreducible generalized Heegaard splittings on the one hand and bridge
position and thin position of knots deserves to be investigated more
closely.  Suffice it to say that this correspondence played a role in
the discovery of the argument yielding the inequality below.  The lack
of degeneracy for tunnel numbers of small knots is mirrored by their
lack of degeneracy of width.  Nevertheless, more generally, tunnel
numbers do degenerate under connected sum and so might their widths.
M Scharlemann and A Thompson conjecture that there are knots whose
width remains constant under connected sum with a 2--bridge knot.  See
\cite{ST1}.

Finally, a lower bound on the width of the connected sum in terms of
the widths of the summands was established by Scharlemann and Schultens:

\begin{thm}[Scharlemann--Schultens \cite{SS}] 
\label{SS}
For any two knots $K_1, K_2$, \[w(K_1 \# K_2) \geq
max(w(K_1), w(K_2))\]
\end{thm}

\begin{cor}[Scharlemann--Schultens \cite{SS}] 
\label{SS1}
For any two knots $K_1, K_2$, \[w(K_1 \# K_2) \geq
\tfrac{1}{2}(w(K_1) + w(K_2))\]
\end{cor}

The fact that width of 3--manifold behaves well from many points of
view has initiated reconsiderations of the notion of thin position
for knots.  One is tempted to redefine the notion of thin position
for knots so as to avoid the difficulties it engenders.
Scharlemann and Thompson have defined a notion of ``slender
knots'' which lies outside of the scope of this article.

\section{Work of Heath and Kobayashi}
\label{sec:hk}

D\,J Heath and T Kobayashi were the first to use the possible
compressibility of thin levels to advantage.  They made great
strides in understanding many issues related to thin position of
knots.  In this section we briefly summarize their results.
Details on these results may be found in three of their joint
papers \cite{HK1,HK2,HK3}.  The illustrations alone are each worth
a thousand words.  Our brief summary requires a number of
definitions. Some of these are analogous to other definitions in
this survey, but are given here in a slightly different context.

Given a link $L$ in ${\mathbb S}^3$, our height function, $h(x)$,
can be thought of as resulting from looking at ${\mathbb S}^3-2
\mbox{ points } = {\mathbb S}^2 \times {\mathbb R}$.  This restriction of
$h(x)$ to ${\mathbb S}^2 \times {\mathbb R}$ is then simply
projection onto the ${\mathbb R}$ factor. We let $p(x)$ be the
projection onto the ${\mathbb S}^2$ factor. Consider a meridional
2--sphere $S$, that is, a 2--sphere in ${\mathbb S}^3$ that
intersects $L$ in points.  (In the complement of $L$, the remnant
of $S$ has boundary consisting of meridians.)

The 2--sphere $S$ is said to be {\em bowl like} if all of the
following hold (see \fullref{FigBowlLike}):
\begin{enumerate}
\item $S = F_1 \cup F_2$ and $F_1 \cap F_2 = \partial F_1 =
\partial F_2$;
\item $F_1$ is a 2--disc contained in a level plane;
\item $h |_{F_2}$ is a Morse function with exactly one maximum or
minimum;
\item $p(F_1) = p(F_2)$;
\item $p|_{F_2}:F_2 \rightarrow p(F_2)$ is a homeomorphism;
\item all points of intersection with $L$ lie in $F_1$.
\end{enumerate}
A bowl like 2--sphere is {\em flat face up} ({\em flat face down})
if $F_1$ is above (below) $F_2$ with respect to $h$.

\begin{figure}[ht!]
\centerline{\includegraphics[width=.4\textwidth]{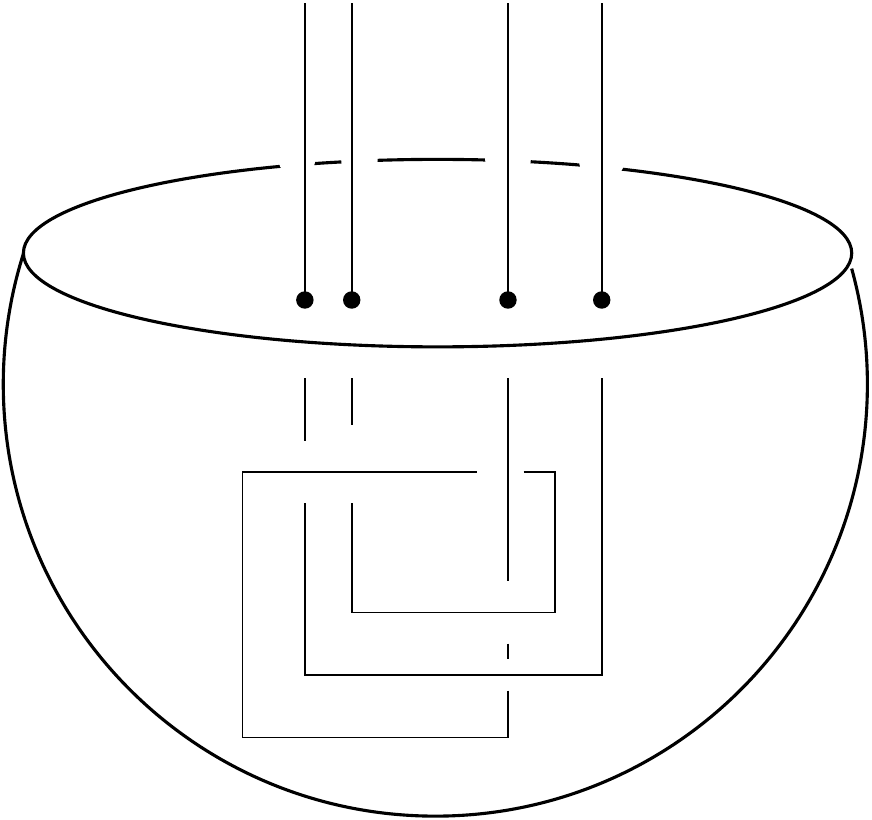}}
\caption{A bowl like 2--sphere (flat face up)}
\label{FigBowlLike}
\end{figure}

Let $F = h^{-1} (r)$, for some regular value $r$ of $h$, be a
thick 2--sphere for $L$. Let $N_0$ ($N_1$) be a thin 2--sphere lying
directly above (below) $F$. Let $D$ be a disc such that $\partial
D = \alpha \cup \beta$ with $\alpha$ a subarc of $L$ containing a
single critical point which is a maximum, and $\beta = D \cap F$.
Similarly, let $D'$ be a disc such that $\partial D' = \alpha'
\cup \beta'$ with $\alpha'$ a subarc of $L$ containing a single
critical point which is a minimum, and $\beta' = D' \cap F$.
Assume that the interiors of $\beta$ and $\beta'$ are disjoint and
that $\alpha \cup \alpha'$ is not a complete component of $L$.
Then $D$ and $D'$ are called {\em a bad pair of discs}.  They are
called {\em a strongly bad disc pair} if $D \cap N_0 = \emptyset =
D' \cap N_1$.

Let $S= S_1 \cup \dots \cup S_n$ be a collection of bowl like 2--spheres for a
link $L$ and $C_0, \dots C_n$ be the closure of the components of $S^3 - S$.
Note that each $S_i$ separates ${\mathbb S}^3$ into two sides.  The one not
containing the 2 points that have been removed from ${\mathbb S}^3$ is
considered to lie inside $S_i$. For $i = 1, \dots, n$, $C_i$ is a punctured
copy of the 3--ball lying inside $S_i$. Furthermore, $C_0$ is the component
which does not lie interior to any $S_i$, and $C_i$ $(i = 1, . . . ,m)$ is the
component lying directly inside of $S_i$.

Let $L_i = L \cap C_i$.  We define thin (thick) level disk analogously to thin
(thick) level spheres.  We say that $L_i, i \neq 0$, is in {\em bridge
position} if there exists some thick 2--disk $D_i \subset C_i$ for $L_i$ such
that all maxima (minima) of $L_i$ are above (below) $D_i$, and every flat face
down (up) bowl like 2--sphere $S_j$ contained in the ``inner boundary'' of $C_i$
(where $C_i$ meets $S_j$ for $j \neq i$) is above (below) $D_i$.  We say that
{\em $L_0$ is in bridge position} if there exists some thick 2--sphere $D_0
\subset C_0$ for $L_0$ having the analogous properties. Finally, let $L'$ be a
portion of the link $L$ lying inside the bowl like 2--sphere $S$. We also say
that {\em $L'$ is in bridge position} if there exists some thick 2--disk $D$
for $L'$ such that all maxima (minima) of $L'$ are above (below) $D$.

We wish to associate a graph with the above information.  To this
end we will suppose that, given $S$ and $C_0, \dots, C_n$ as
above, the following properties are satisfied.  If they are, the
system is said to enjoy {\em Property 1}.\\

\begin{enumerate}
\item For each $C_j$, $(j = 0, 1, . . . ,m)$, we have one of the
following:
\begin{enumerate}
\item there are both a maximum and a minimum of $L$ in $C_j$; or
\item there does not exist a critical point of $L$ in $C_j$.
\end{enumerate}
\item There exists a level 2--sphere $F_0$ in $C_0$ such that both
of the following hold:
\begin{enumerate}
\item every flat face down (up, respectively) bowl like 2--sphere
in $\partial C_0$ lies above (below, respectively) $F_0$; and
\item every maximum (minimum, respectively) of $L$ in $C_0$ (if
one exists) lies above (below, respectively) $F_0$, and it is
lower (higher, respectively) than the flat face down (up,
respectively) bowl like 2--spheres in $\partial C_0$.
\end{enumerate}

\item For each $i$, $(i = 1, . . . ,m)$, there exists a level disk
$F_i$ properly embedded in $C_i$ such that both of the following
hold:
\begin{enumerate}
\item every flat face down (up, respectively) bowl like 2--sphere
in $\partial C_i$ lies above (below, respectively) $F_i$ , and
\item every maximum (minimum, respectively) of $L$ in $C_i$ (if
one exists) lies above (below, respectively) $F_i$, and it is
lower (higher, respectively) than the face down (up, respectively)
bowl like 2--spheres in $\partial C_i - S_i$.
\end{enumerate}
\end{enumerate}

A {\em spatial graph} $G$ is a 1--complex embedded in the 3--sphere.
$G$ is a {\em signed vertex graph} if each vertex of $G$ is
labeled with either a $+$ or a $-$. The {\em width of $G$} is
defined as follows. Suppose that the vertices of $G$ labeled with
$+$ ($-$, respectively) have the same height and are higher
(lower, respectively) than any other point in $G$.  Suppose
further that $h|_{G - \{vertices\}}$ is a Morse function.  We say
that $G$ is in {\em bridge position} if each maximum in $G -
\{vertices\}$ is higher than any minimum of $G - \{vertices\}$. In
general, let $r_1, \dots , r_{n-1}$ $(r_1 < \dots < r_{n-1})$ be
regular values between the critical values in $G - \{vertices\}$.
Then define the width of $G$ to be the following $w(G) =
\Sigma_{i=1}^{n-1} | G \cap h^{-1}(r_i)|$. For a signed vertex
graph in bridge position {\em the bridge number} is $|F \cap G| /
2$ where $F$ is a level 2--sphere such that every maximum of $G$ is
above $F$ and every minimum of $G$ is below $F$. The minimum of
the bridge numbers for all possible bridge positions of $G$ is the
{\em bridge index} of $G$.

\begin{figure}[ht!]
\centerline{\includegraphics[width=.5\textwidth]{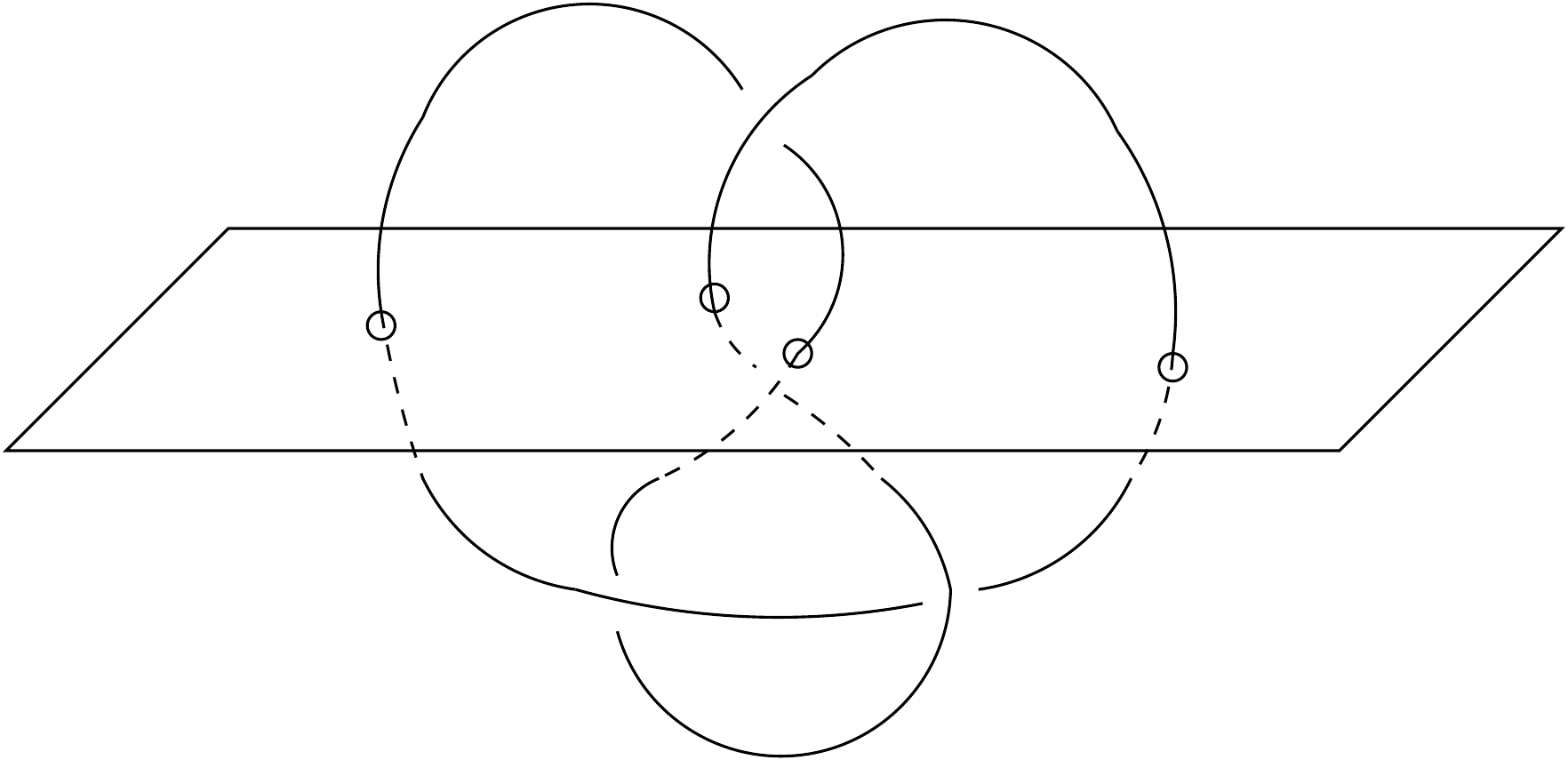}}
\caption{The Figure 8 knot in bridge position}
\label{fig8}
\end{figure}

Let $L$, $C_j$ $(j = 0, 1, . . . ,m)$ be as above. We can obtain a
signed vertex graph $G_j$ from ($C_j ,L \cap C_j$) as follows: In
the case that $j = 0$, shrink each component of $\partial C_0$ to
a vertex.  Then pull up (down, respectively) the vertices obtained
from flat face down (flat face up, respectively) 2--spheres so that
they lie in the same level. We obtain the signed vertex graph
$G_0$ by assigning $+$ to the former and $-$ to the latter. By (2) of
Property 1, we see that $G_0$ is in a bridge position.

Suppose that $j \neq 0$. In this case we may deform $C_j$ by an
ambient isotopy, $f_t$, of ${\mathbb S}^3$ which does not alter
the flat face of $S_j$ so that $f_1(C_j)$ appears to be of type
$C_0$.  This isotopy moves infinity ``into'' $C_j$.  For details see
for instance the ``Popover Lemma'' in \cite{Sc3}. Then we apply the
above argument to $(f_1(C_j), f_1(L \cap C_j))$ and obtain a
signed vertex graph $G_j$ in bridge position. We say that $G_j$
$(j = 0, 1, . . . ,m)$ is a {\em signed vertex graph associated to
S}.  In this process we reversed $S_j$ and made no other changes; thus the
resulting signed graph is the same as the signed graph in the case $j=0$ but
the sign of the vertex corresponding to $S_j$ is reversed.

Let $L, C_j$ $(j = 0, 1, . . . ,m)$ be as above. Then we can take
a convex 3--ball $R_j$ in the interior of $C_j$ such that each
component of $(L \cap C_j) - R_j$ is a monotonic arc connecting
$R_j$ and a component of $\partial C_j$, and such that
\begin{enumerate}
\item $R_0$ lies below (above, respectively) the flat face down (up,
respectively) bowl like 2--spheres in $\partial C_0$;
\item $R_i$ $(i = 1, . . . ,m)$ lies below (above respectively) the
flat face down (up respectively) bowl like 2--spheres in $\partial C_i
- S_i$.
\end{enumerate}
We call $R_j$ a {\em cocoon} of $L$ associated to $S$.

\subsection{A search method for thin position of links}

Let $L$ be a link of bridge index $n$ and suppose that there is a
list of all those meridional, essential, mutually non parallel
planar surfaces in the exterior of $L$, that have at most $2n-2$
boundary components.  Let $S = \bigcup_{i=1}^m S_i$ be a union of
2--spheres in $S^3$ as above. Then we can obtain a number of
systems of signed vertex graphs as follows: For each $i,$ $(i = 1,
. . . ,m)$, we assign $+$ to one side of $S_i$ and $-$ to the other.
Note that there are $2^m$ ways to make such assignments. Let
$C_0,C_1, \dots ,C_m$ be as above. Then for each $j,$ $(j = 0, 1,
. . . ,m)$ the collar of each component of $\partial C_j$ is
assigned either a $+$ or a $-$. By regarding each component of
$\partial C_j$ as a very tiny 2--sphere, we obtain a signed vertex
graph, say $G_j$, from $L \cap C_j$.

Now we assume, additionally, that we know the bridge indices of
all the signed vertex graphs obtained in this manner. Then, for
each system of signed vertex graphs, we take minimal bridge
presentations, say $G_0,G_1, . . . ,G_m$, of the signed vertex
graphs. We expand the vertices of $G_0,G_1, . . . ,G_m$ to make +
vertices (- vertices, respectively) flat face down (up,
respectively) bowl like 2--spheres. Then we combine the pieces,
applying the inverse of deformations  to obtain a position of $L$,
say $L'$, and a union of bowl like 2--spheres $S'$ with respect to
which $L'$ satisfies Property 1 above. Let $R_0,R_1, . . . ,R_m$
be the cocoons of $L'$ associated to $S'$. Then consider all
possible orders on $\{ R_0,R_1, . . . ,R_m \}$ which are
compatible with relative positions in $L'$. All such orders are
realized as a position of L.

\begin{thm}[Heath--Kobayashi {{\cite[Theorem~2]{HK3}}}]
There is a thin position of $L$ that is realized
through the process described above.
\end{thm}

\subsection{Essential tangle decomposition from thin position of a link}

We say that two links $L$ and $L'$ in ${\mathbb S}^3$ are {\em
h-equivalent} if there exists an ambient isotopy, $f_t$, such
that $f_0(L) = L$, $f_1(L) = L'$ and such that for every $x \in L$
we have $h(f_1(x)) = h(x)$.

\begin{prop}[Heath--Kobayashi {{\cite[Proposition~3.7]{HK1}}}] 
If a link $L$ has the property that thin position
differs from bridge position, then there exists an ambient isotopy
$f_s$, such that $L' = f_1(L)$ is h-equivalent to $L$ and $L'$ has
a tangle decomposition by a finite number of non-trivial,
non-nested, flat face up, bowl like 2--spheres, each of which is
incompressible in the link complement.  In this decomposition we
have a tangle ``on top'' (above $P$) with all of the
incompressible 2--spheres below it connected by vertical strands.
\end{prop}

\begin{thm}[Heath--Kobayashi {{\cite[Theorem~4.3]{HK1}}}]
Let $L$ be a link in thin position, and $S$
as above. Then there exists an ambient isotopy for $L$ to a link
$L'$ so that there exists a collection of incompressible bowl like
2--spheres $S'$ for $L'$ such that there is a one to one
correspondence between the components of ${\mathbb S}^3 - S'$ that
contain maximum (and minimum) of $L'$ and the components of
${\mathbb S}^3 - S$ that contain maximum (and minimum) of $L$.
\end{thm}

\subsection{Locally thin position for a link}

Perhaps the greatest weakness of thin position, as with many knot
invariants that are defined in terms of a global minimum, is that it is 
hard to determine. On the other hand, many of its applications
rely only on local properties of thin position.  In order to
address this issue, Heath and Kobayashi define a local version of
thin position.

\begin{defn}
A link $L$ is said to be in local thin position if it satisfies
the following
two properties with respect to the height function $h$:
\begin{enumerate}
\item no thick 2--sphere for $L$ has a strongly bad pair of discs, and
\item There exists a decomposition of $L$ with bowl like 2--spheres
$S_1, \dots, S_n$ such that each $S_i$ is incompressible and
$\partial$--incompressible, and so that $L$ is in bridge position
in the complement of $\cup S_i$.
\end{enumerate}
\end{defn}

They then prove the following main result and two corollaries:

\begin{thm}[Heath--Kobayashi {{\cite[Main Theorem~3.1]{HK2}}}]
Every non-splittable link has a locally thin presentation.
\end{thm}

\begin{cor}[Heath--Kobayashi {{\cite[Corollary~3.4]{HK2}}}]
Any locally thin position of the unknot is trivial.
\end{cor}

\begin{cor}[Heath--Kobayashi {{\cite[Corollary~3.5]{HK2}}}]
Any locally thin position of a 2--bridge knot is in 2--bridge
position.
\end{cor}

These corollaries both show the strength and the weakness of local
thin position. It is not as easy to compute as one might hope or
this would mean that recognizing the unknot and 2--bridge knots
would be easy, but in exchange it contains significant information
when it is computed.

\section{Compressibility of thin levels}
\label{sec:compressibility}

Ying-Qing Wu began an investigation of the thin levels for knots in
thin position.  He proved the following about the thinnest thin level,
that is, the thin level that meets the knot in the fewest number of
points.

\begin{thm}[Wu \cite{W}] 
If $K$ is in thin position with respect to $h$, then the thinnest
thin level of $K$ is incompressible.
\end{thm}

Wu's strategy is to show that if a thin level is compressible, then
the surface obtained by compressing it is parallel to another thin
level.  His result then follows by induction.  He also demonstrates
applications of this result: He uses it to give an alternative
proof of the Rieck--Sedgwick Theorem.

Maggy Tomova continued this investigation in \cite{To}.  She proved
more refined results about compressing disks for thin levels of links
in thin position.  Her results rely on a number of concepts,
observations and lemmas.  We give a very brief overview, for details
see \cite{To}.  In particular, note that the description below relies
on many technical lemmas.

Suppose the link $L$ in ${\mathbb S}^3$ is in thin position.  Further
suppose that $P = h^{-1}(r)$ is a thin level for $K$ and that $D$ is a
compressing disk for $P$ in the complement of $K$.  We may assume that
the interior of $D$ lies entirely above or entirely below $P$, say,
the former.  To help our visualization of the situation, we imagine
$D$ as a cylinder lying vertically over $\partial D$ and capped off
with the maximum, $\infty$, of $h$.  It is then clear that $D$
partitions the portion of $L$ lying above $P$ into two subsets.
Denote the portions of $L$ above $P$ that are separated by $D$ by
$\alpha$ and $\beta$.  See \fullref{ab}.

\begin{figure}[ht!]
\centerline{\includegraphics[width=.7\textwidth]{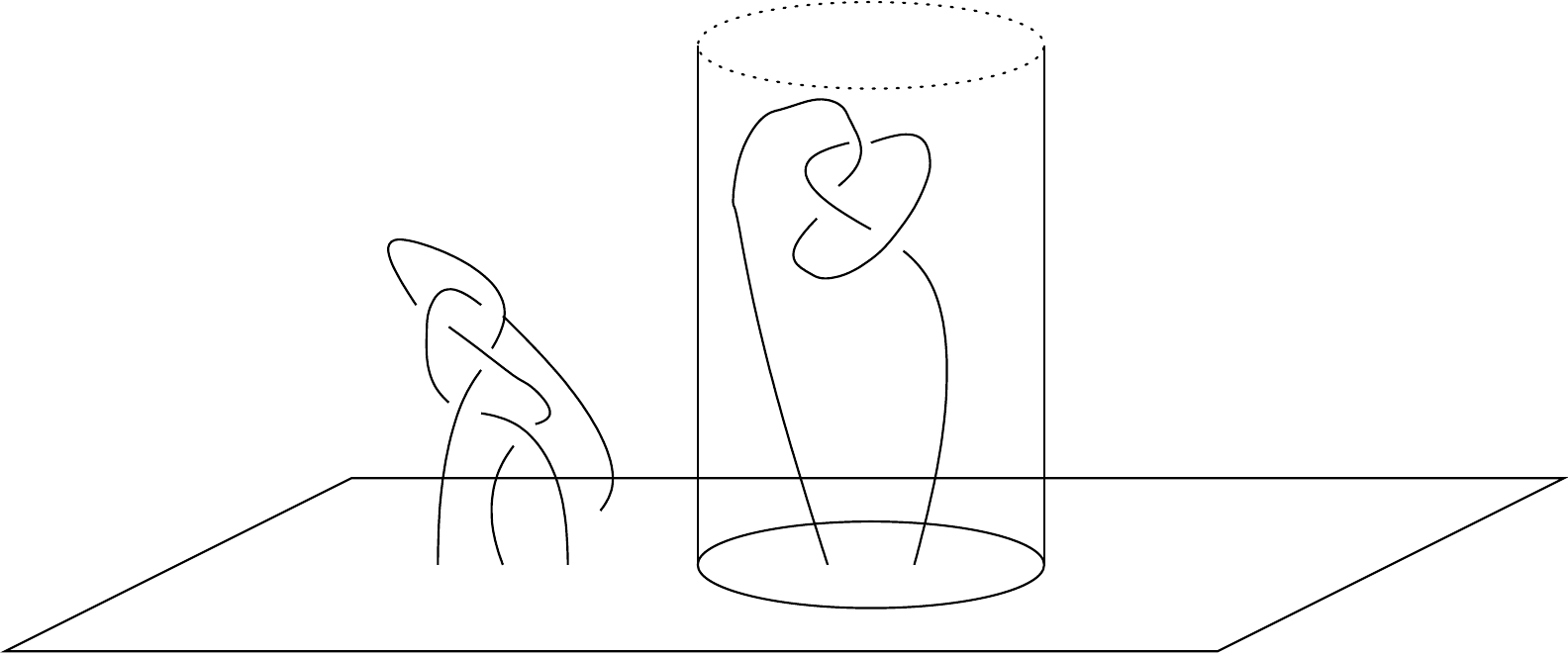}}
\caption{The portions $\alpha$ and $\beta$ of $L$}
\label{ab}
\end{figure}

Now play off $\alpha$ versus $\beta$.  An alternating thin level is a
thin level $P' = h^{-1}(r')$ above $P$ such that the first minimum
above $P'$ lies on $\alpha$ and the first maximum below $P'$ lies on
$\beta$ or vice versa.  As it turns out, alternating thin levels
necessarily exist; furthermore, for any adjacent alternating thin
levels, either the portion of $\alpha$ or the portion of $\beta$ lying
between the two alternating thin levels is a product.

Interestingly, if we number the alternating thin levels above $P$ by
$A_1,$ $\dots,$ $A_n$, such that $h(A_{j-1}) < h(A_j)$, then the
sequence $w_1, \dots, w_n$ defined by $w_j = \#|K \cap A_j|$ is
strictly decreasing.  The class of alternating thin levels can be
enlarged to include other thin levels that satisfy certain technical
properties enjoyed by alternating thin levels.  The resulting class of
surfaces are the potentially alternating surfaces.  Compressing disks
such as $D$ can then be assigned a height: Assign $D$ the height $k$
if $D \cap A_{k-1} \neq \emptyset$ but $D \cap A_k = \emptyset$.

The short ball for a compressing disk $D$ for $P$ is the ball bounded
by $D$ and a subdisk of $P$ that contains the shorter of $\alpha$ or
$\beta$, that is, that portion of the knot whose absolute maximum is
lower than that of the other.  (By transversality, these two maxima do
not lie on the same level.)  The compressing disk $D$ for $P$ is
reducible if there is a disk $E$ whose interior lies in the short ball
for $D$, whose boundary is partitioned into an arc $\tau$ on $D$ and
an arc $\omega$ on $P$ and for which $\omega$ is essential in $P -
(\partial D \cup L)$.  A compressing disk is irreducible if it is not
reducible.

A key result is the following:

\begin{thm}[Tomova \cite{To}] 
Suppose $D$ and $D'$ are two irreducible compressing disks
for $P$, and $\alpha, \alpha'$ are the strands of $L$ lying in the
corresponding short balls.  Then $\height(D) = \height(D')$ implies
$\alpha = \alpha'$.  Otherwise, $\alpha \cap \alpha' = \emptyset$.
\end{thm}

\begin{cor}[Tomova \cite{To}] 
Any two distinct irreducible compressing disks for $P$ of the
same height must intersect.
\end{cor}

See \fullref{DD'}.

\begin{figure}[ht!]
\centerline{\includegraphics[width=.7\textwidth]{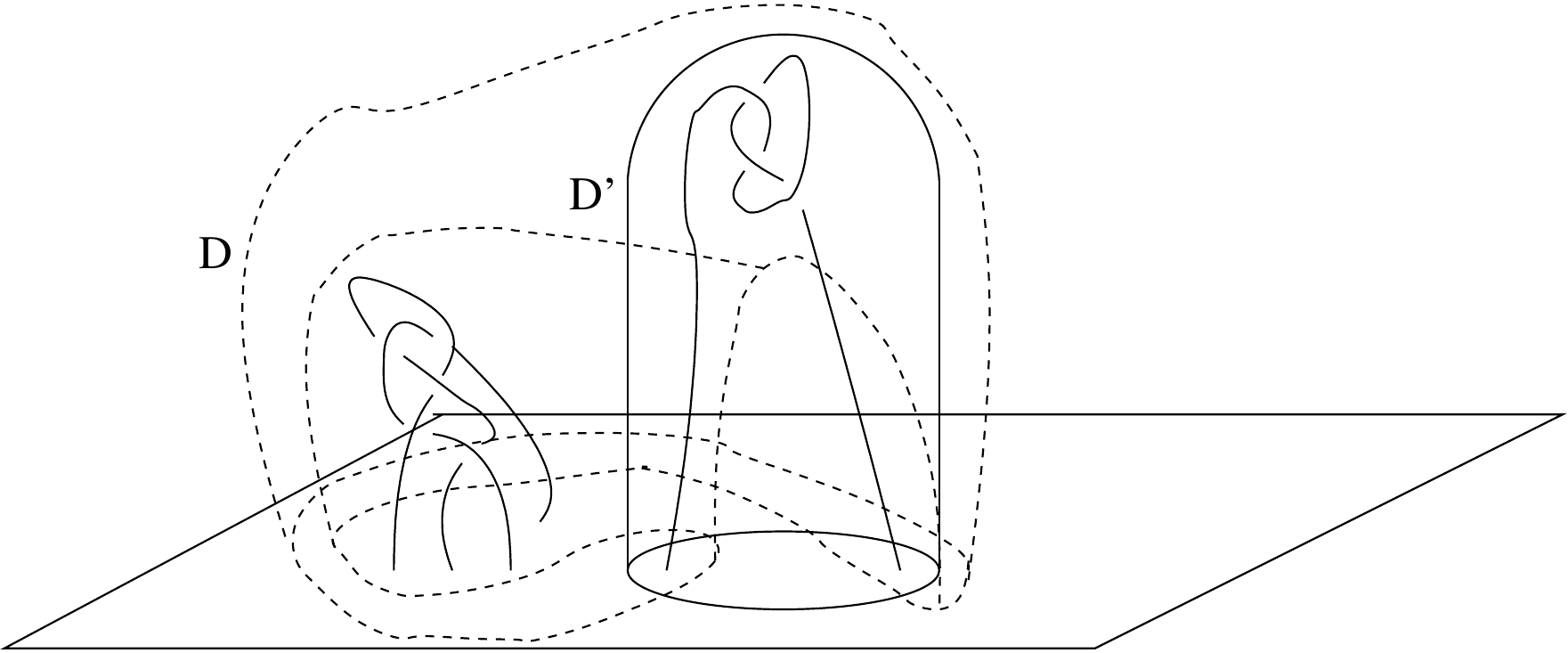}}
\caption{Two disks of the same height}
\label{DD'}
\end{figure}

\begin{thm}[Tomova \cite{To}] 
There exists a collection of disjoint irreducible compressing
disks for $P$ that contains one representative from each possible
height.
\end{thm}

In certain situations, these results suffice to guarantee unique
compressing disks for thin levels!

\section{2--fold branched covers}

Let $K$ be a knot in ${\mathbb S}^3$ and let $M$ be the 2--fold branched
cover of ${\mathbb S}^3$ over $K$.  It is natural to ask the following
question: How is thin position of $K$ related to thin position of $M$?
This question is investigated by Howards and Schultens \cite{HS}.
(A related question about the behavior of the Heegaard genus under double
covers was investigated by Rieck and Rubinstein \cite{RR}).

A height function on ${\mathbb S}^3$ lifts to a Morse function on $M$.
The thick and thin levels of $K$ and $M$ are related as follows:
Denote the thick levels of $K$ by $S_1, \dots, S_n$ and the thin
levels by $L_1, \dots, L_{n-1}$.  Here the $S_i$'s and $L_i$'s are
spheres that meet the knot some (even) number of times.  In fact,
each $S_i$ meets $K$ at least $4$ times and each $L_i$ meets $K$ at
least $2$ times.  Denote the surface in $M$ corresponding to $S_i$ by
$\tilde S_i$ and the surface in $M$ corresponding to $L_i$ by $\tilde
L_i$.  If $S_i$ meets $K$ exactly $2l$ times, then $\tilde S_i$ is a
closed orientable surface of genus $l-1$.  And if $L_i$ meets $K$
exactly $2l$ times, then $\tilde L_i$ is a closed orientable surface
of genus $l-1$.

Compressing disks for $\tilde S_i$ may be constructed by taking a disk
$D$ in ${\mathbb S}^3$ that is disjoint from $L_1, \dots, L_{n-1}$,
whose interior is disjoint from $S_1, \dots, S_n$ and whose boundary
is partitioned into an arc $a$ in $S_i$ and an arc $b$ in $K$ that has
exactly one critical point.  (Such a disk is called a strict
upper/lower disk.)  The \dbc \hspace{.75 mm} $\tilde D$ of $D$ has its
boundary on $\tilde S_i$ and is a compressing disk for $\tilde S_i$.
This illustrates the fact that $\tilde L_{i-1}$ and $\tilde S_i$ and
also $\tilde L_i$ and $\tilde S_i$ cobound compression bodies.

Now if $K$ is in thin position, then one may ask whether or not the
\mdc \hspace{.75 mm} that $M$ inherits is in thin position.

\begin{thm}[Howards--Schultens]
If $K$ is a 2--bridge knot or a 3--bridge knot, then the \mdc
\hspace{.75 mm} that $M$ inherits is in thin position.
\end{thm}

This result is not true for knots in general.  Consider for instance
torus knots.  For torus knots the \mdc \hspace{.75 mm} that their \dbc
\hspace{.75 mm} inherits is not necessarily in thin position.  To see
this, consider the following: The complement of a torus knot is a
Seifert fibered space fibered over the disk with two exceptional
fibers.  This places restrictions on the type of incompressible
surfaces that can exist.  In particular, it rules out meridional
surfaces.  For a discussion of incompressible surfaces in Seifert
fibered spaces, see for instance Hempel \cite{H} or Jaco \cite{J}.  It follows
that $K$ is mp-small.

Now Thompson's Theorem (\fullref{abby}) implies that thin position
for $K$ is bridge position.  Bridge numbers for torus knots can be
arbitrarily large.  Specifically, if $K$ is a $(p,q)$--torus knot, then
the bridge number of $K$ is $\min\{p,q\}$.  This was proved by Schubert
in \cite{S}.  For a more contemporary and self contained proof see
Schultens \cite{Sc5}.  Thus for $K$ in thin position, the \mdc \hspace{.75 mm}
that the \dbc \hspace{.75 mm} inherits is a Heegaard splitting of
genus $\frac{\min\{p,q\}}{2} - 1$.

On the other hand, the \dbc \hspace{.75 mm} of ${\mathbb S}^3$ over a torus knot is a
small Seifert fibered space.  Specifically, the \dbc \hspace{.75 mm} of ${\mathbb
S}^3$ over the $(p, q)$--torus knot is a Seifert fibered space fibered
over ${\mathbb S}^2$ with three exceptional fibers of orders $p, q,
2$.  (Such manifolds are also called Brieskorn manifolds.)  But any
such manifold possesses Heegaard splittings of genus $2$.

\section{Questions}

The following questions deserve to be considered:
\begin{enumerate}
\item  Develop a less unwieldy notion of thin position for knots.

\item Find an algorithm to detect the width of a knot.  In light of the
discussion at the end of \fullref{sec:hk} we ask: find an algorithm
to place a knot in local thin position for knots.

\item Characterize the compressibility of thin levels for knots in
thin position.

\item Construct knots of arbitrarily large width.

\item Apply the concept of thin position in completely different settings.
\end{enumerate}

\bibliographystyle{gtart}
\bibliography{link}

\end{document}